\documentclass{amsart}

\usepackage[latin1]{inputenc}
\usepackage{color}
\usepackage{amssymb}

\newtheorem{theorem}{Theorem}[section]
\newtheorem*{theorem*}{Theorem}
\newtheorem{lemma}[theorem]{Lemma}

\newtheorem{definition}{Definition}

\makeatletter

\newcommand{\LL}{\mathcal{L}}
\newcommand{\zz}{\mathbb{Z}}

\newcommand{\hind}{h_{\mathit{ind}}}

\newcommand{\Xinf}{X^{\otimes \infty}}
\newcommand{\Xmax}{\widehat{X^{\otimes \infty}}_{\max}}

\title[On independence and entropy for isotropic subshifts]{On independence and entropy for high-dimensional isotropic subshifts}

\date{}
\author{Tom Meyerovitch}
\address{Tom Meyerovitch\\
Department of Mathematics\\
Ben-Gurion University of the Negev}
\email{mtom@math.bgu.ac.il}
\author{Ronnie Pavlov}
\address{Ronnie Pavlov\\
Department of Mathematics\\
University of Denver}
\email{rpavlov@du.edu}


\begin{document}

\begin{abstract}
In this work, we study the problem of finding the asymptotic growth rate of the number of of $d$-dimensional arrays with side length $n$ over a given alphabet which avoid a list of one-dimensional ``forbidden'' words along all cardinal directions, as both $n$ and $d$ tend to infinity. Louidor, Marcus, and the second author called this quantity the ``limiting entropy''; it is the limit of a sequence of topological entropies of a sequence of isotropic $\mathbb{Z}^d$ subshifts with the dimension $d$ tending to infinity \cite{LMP}. We find an expression for this limiting entropy which involves only one-dimensional words, which was implicitly conjectured in \cite{LMP}, and given the name ``independence entropy.'' In the case where the list of ``forbidden'' words is finite, this expression is algorithmically computable and is of the form $\frac{1}{n} \log k$ for $k,n \in \mathbb{N}$. Our proof also characterizes the weak limits (as $d \rightarrow \infty$) of isotropic measures of maximal entropy; any such measure is a Bernoulli extension over some zero entropy factor from an explicitly defined set of measures. We also demonstrate how our results apply to various models previously studied in the literature, in some cases recovering or generalizing known results, but in other cases proving new ones.

The core idea of our proof is to consider certain isotropic measures on a limiting ``infinite dimensional'' subshift, and apply the classical theorem of de Finetti on exchangeable random variables.
\end{abstract}

\maketitle

\section{Introduction}

We start with a simple, concrete problem to illustrate our main result:
How many $d$-dimensional arrays of side length $n$ of letters from the English alphabet are there in which the word ``ADD'' does not appear along any cardinal direction?
Let us denote this number by $B_{n,d}$. In general, there are no known closed-form expressions for such quantities for $d > 1$.
What about the limit $B_d := \lim_{n \to \infty} \frac{1}{n^d} \log B_{n,d}$ for a fixed $d > 1$?
It turns out that again there is generally no general closed-form expression for $d > 1$.

However, a direct application of our main result gives:
$$\lim_{d,n \to \infty} \frac{1}{n^d} \log B_{n,d} = \frac{1}{2}\log 25 + \frac{1}{2}\log 26.$$
Our result also gives information about the limit of uniform distributions over such arrays, again as $n,d \rightarrow \infty$.
The general framework within which we state and prove our results is the field of \emph{symbolic dynamics}.

Symbolic dynamics is concerned with
the study of particular $\mathbb{Z}^d$ topological dynamical systems
called subshifts. A $\mathbb{Z}^d$ subshift is defined by a finite
set $\Sigma$, called an {\it alphabet}, and a (possibly infinite) set
$\mathcal{F}$ of functions from finite subsets of
$\mathbb{Z}^d$ to $\Sigma$. The elements of $\mathcal{F}$ are called \textit{forbidden configurations}. The
$\mathbb{Z}^d$ subshift with alphabet $\Sigma$ induced by
$\mathcal{F}$, denoted by $X(\mathcal{F})$, is defined to be the set
of points in $\Sigma^{\mathbb{Z}^d}$ in which no translate of a (forbidden) configuration from
$\mathcal{F}$ appears. Equivalently, a $\mathbb{Z}^d$ subshift is a subset of $\Sigma^{\mathbb{Z}^d}$ which is compact with respect to the product topology and invariant under the \emph{shift} action of $\mathbb{Z}^d$.
In the special case where $\mathcal{F}$ is finite,
$X$ is called a $\mathbb{Z}^d$ \emph{shift of finite type}, or SFT.

In this paper we obtain results describing the asymptotic behavior
of high-dimensional \emph{isotropic} $\zz^d$ subshifts;
we use the word ``isotropic'' here to refer to an
object (e.g. subshift, measure) which exhibits the same behavior along each cardinal direction.
Our results mostly concern various notions of entropy for such subshifts.
 
Given a one-dimensional subshift $X = X(\mathcal{F}) \subset \Sigma^{\mathbb{Z}}$, the set of possible points in $\Sigma^{\mathbb{Z}^d}$
for which words from $\mathcal{F}$ do not occur in any row along any cardinal direction is a $d$-dimensional subshift, which was
called the $d$th \emph{axial power} of $X$ in \cite{LMP}, and which we denote by $X^{\otimes d}$. For instance, for $\Sigma = \{0,1,2\}$ and $\mathcal{F} = \{00,11,22\}$,
$X^{\otimes d}$ contains all $\{0,1,2\}$-colorings of $\zz^d$ where adjacent sites (sites with distance $1$) have distinct colors.

The \emph{topological entropy}, which we will review in Section \ref{defns}, is a useful quantity associated to any $\mathbb{Z}^d$ subshift which, in some sense, measures its complexity. Our main result concerns, for any one-dimensional subshift, the limit as $d \rightarrow \infty$ of topological entropies of $\mathbb{Z}^d$ of its axial powers, which we call the \emph{limiting entropy} of $X$ and denote by $h_{\infty}(X)$. 
 We will show (Lemma~\ref{lem:hinf_hXinf}) that $h_{\infty}(X)$ can also be viewed as the
topological entropy of a suitably defined ``infinite-dimensional'' axial power; 
this observation will be essential to all of our arguments.
In \cite{LMP}, connections were made between the limiting entropy and another quantity, called independence entropy.

The \emph{independence entropy} of a one-dimensional subshift $X$ (denoted by $\hind(X)$)
is, informally speaking, a measure of how much of the
topological entropy of $X$ comes from sitewise independent behavior.
If a one-dimensional subshift $X$ has alphabet $\Sigma$, then we say
that a string $A_1 A_2 \ldots A_n$ of nonempty subsets of $\Sigma$
is ``independently legal for $X$'' if $a_1 \ldots a_n$ appears in a
point of $X$ for every choice of $a_1 \in A_1, \ldots, a_n \in A_n$.

For example, consider the one-dimensional subshift of finite type
consisting of all $x \in \{0,1\}^{\mathbb{Z}}$ which do not contain
consecutive $1$s, often called the \emph {golden mean shift} or the \emph {one-dimensional hard-square model}, denoted here by $X_G$.
Then $\{0,1\} \{0\} \{0,1\}$ is independently legal for $X_G$, since
all four words $000, 001, 100, 101$ appear in points of $X_G$.
However, $\{0,1\} \{0,1\} \{0\}$ is not independently legal for $X_G$,
since the word $110$ is illegal in $X_G$. (This is because $110$
contains consecutive $1$s.)

Any independently legal string of subsets for a subshift $X$ can be
thought of as a source of words appearing in points of $X$ which are
induced by sitewise independent choices. The independence entropy of
$X$ is defined as the asymptotic exponential growth rate (in $n$) of the
maximum number of words in $X$ so induced by a single $n$-letter
independently $X$-legal string. For example, in the case of $X_G$ above, it is easy to see that
the independently $X_G$-legal strings which induce the most legal words
in $X_G$ are words of the form $\{0,1\} \{0\} \{0,1\} \{0\} \ldots$ which
alternate between $\{0,1\}$ and $\{0\}$. For any $n$, such a string induces
$2^{\lceil 0.5n \rceil}$ words in $X_G$, and so the independence entropy of
$X_G$ is $\frac{\log 2}{2}$. 

Independence entropy can be defined in a similar way for any $\mathbb{Z}^d$ subshift, and
in \cite{LMP} it was shown that $\hind(X^{\otimes d}) = \hind(X)$ for any $X$ and $d$. This was then
used to show that $\hind(X) \leq h_{\infty}(X)$ for all subshifts $X$. It was also
implicitly conjectured in \cite{LMP} that the two quantities $\hind(X)$ and $h_{\infty}(X)$
are always equal. Our main result is that this conjecture is true.

\begin{theorem}\label{mainresult}
\label{hind=hinf} For any $\zz$ subshift $X$, the limiting entropy of $X$ is equal to the independence entropy of $X$.
\end{theorem}

For any $\zz$ SFT $X$, it was shown in \cite{LMP} that $\hind(X)$ is of the form $\frac{\log k}{n}$ where $k,n \in \mathbb{N}$,
and that there is an simple algorithm to compute $\hind(X)$. Thus, an important consequence of our result is that the limiting entropy for subshifts of finite type can be easily computed, and is always a rational multiple of the logarithm of a natural number. This is in sharp contrast to the topological entropy of $d$-dimensional SFTs, for which there is no known explicit expression, even for some of the simplest nontrivial examples.

To give an idea of the structure of the paper, here is a brief outline of the proof. As stated above, we begin by verifying that the limiting entropy of $X$ coincides with the topological entropy of a limiting ``infinite dimensional axial product'' $X^{\otimes \infty}$. The inequality $\hind(X) \leq h_{\infty}(X)$ has already been established,  so we need only prove the reverse inequality. By the variational principle, it is sufficient to show that $\hind(X)$ is greater than or equal to the measure-theoretic entropy for $X^{\otimes \infty}$ with respect to an arbitrary shift-invariant measure. By using the fact that $X^{\otimes \infty}$ is isotropic, we then show that it suffices to consider only shift-invariant measures which are also invariant with respect to finite permutations of the cardinal directions. For such measures, we can consider the conditional measure of states at certain sites given states on certain other sites as random variables, which turn out to be \textit{exchangeable}, and so by de Finetti's theorem they are mixtures of identically distributed random variables. This structure guarantees some independence of these random variables, which can be leveraged into a proof that for almost every point in $X^{\otimes \infty}$, if at each site we associate the set of admissible states given the states of all other sites, the resulting point is ``independently legal'' for $X^{\otimes \infty}$. This naturally defines a measure-theoretic factor map (shift invariant sub-$\sigma$-algebra), with respect to which the relative entropy is bounded from above by the independence entropy of $X^{\otimes \infty}$ (which is equal to the independence entropy of $X$). To conclude the proof, we verify that if our measure is in addition a measure of maximal entropy, the measure-theoretic entropy of this factor is zero.

The remainder of the paper is organized as follows. Section~\ref{defns} contains necessary definitions and preliminary results for our arguments and results. Section~\ref{mainproof} contains the proof of Theorem~\ref{mainresult}. Section~\ref{mmes} contains results regarding measures of maximal entropy for our infinite-dimensional axial power $X$, including a uniqueness criterion. Section~\ref{sec:application} discusses some previously existing results for specific models and how our results fit into this framework, and Section~\ref{future} discusses some natural questions and directions for future work.\\

\textbf{Acknowledgement:} We are grateful to Brian Marcus for his encouragement and for many stimulating discussions, and to Aernout van Enter for explaining some relations of our work to mean-field theory in statistical mechanics. We would also like to thank the math department of the University of British Columbia and the Pacific Institute for Mathematical Sciences for hosting a visit of the second author, when the initial part of this research took place.

\section{Definitions and preliminaries}
\label{defns}

We will consider some actions by homeomorphisms of certain countable groups: $\zz^d$ for $d \in \mathbb{N}$, $\bigoplus_{\mathbb{N}}\mathbb{Z}$, which we denote by $\zz^{\infty}$, and the group of permutations on $\mathbb{N}$ which fix all but finitely many integers, which we denote by $\mathbb{P}$.

By definition, $\mathbb{Z}^{\infty}=\bigoplus_{\mathbb{N}}\mathbb{Z}$ is the countable group
of infinite sequences of integers which have only finitely many non-zero terms. 

All of these groups are \textit{amenable}: each admits a F\o lner sequence, which is any sequence $F_n$ of finite subsets for which
$\frac{|g F_n \triangle F_n|}{|F_n|} \underset{n \rightarrow \infty}{\longrightarrow} 0$ for all $g \in G$. For any $d \in \mathbb{N}$,
$F^{(d)}_n := [-n,n]^d$ is  a F\o lner sequence in $\mathbb{Z}^d$.
For the group $\mathbb{Z}^{\infty}$, an example of a F\o lner sequence is given by $F^{(\infty)}_n = [-n,n]^n \times
\{0\}^{\infty}$. When the dimension is obvious from context, we will refer to these sets simply as $F_n$.

Sometimes it will be useful to view $\mathbb{Z}^d$ ($d \in \mathbb{N} \cup \{\infty\}$) as the vertex set of a graph, where $v$ and $w$ are adjacent
if and only if $\sum_i |v_i - w_i| = 1$. For any subset $S \subseteq \mathbb{Z}^d$, we then define its \textit{boundary} $\partial S$ to
be the set of all vertices in $\mathbb{Z}^d \setminus S$ adjacent to some vertex in $S$.

For any finite set $\Sigma$, which we call an {\it alphabet}, and any $d \in \mathbb{N} \cup \{\infty\}$,
define the {\it $d$-dimensional full shift}
to be the topological dynamical system given by the space $\Sigma^{\zz^d}$, endowed with the
{\it shift $\zz^d$-action} $(\sigma_v)$, defined by $(\sigma_v x)_g := x_{g+v}$ for all $v,g \in \zz^d$ and $x \in \Sigma^{\zz^d}$.
The full shift is endowed with the (discrete) product topology, with respect to which each shift is a homeomorphism.

A {\it configuration} is defined to be any $c \in \Sigma^S$, for
 a finite set $S \subseteq \zz^d$. The set $S$ is called the {\it shape} of the configuration.
For any finite set of configurations $w_1, \ldots, w_n$ with disjoint shapes $S_1, \ldots, S_n$,
their {\it concatenation} is the configuration $w_1 w_2 \cdots w_n$ with shape $\bigcup_{i=1}^n S_i$ defined
by $(w_1 w_2 \cdots w_n)_{S_i} = w_i$ for $1 \leq i \leq n$. For historical reasons, when $d = 1$ a configuration
is referred to as a {\it word}.

A {\it $\zz^d$ subshift} is a subset of the $\zz^d$ full shift which is shift-invariant and closed (therefore compact) with respect to the product topology. Elements of a $\zz^d$ subshift are called {\it points}, and, again for historical reasons, when $d = 1$ we call them {\it sequences.} Any $\zz^d$ subshift $X$ can be defined by a (possibly infinite) set $\mathcal{F}$ of forbidden configurations
in the following way:
\[
X = X(\mathcal{F}) := \{x \in \Sigma^{\zz^d} \ : \  x_{S+n} \notin \mathcal{F} \textrm{ for all finite } S \subseteq \zz^d \textrm{ and } n \in \zz^d\}.
\]
When $\mathcal{F}$ is finite, we say that $X(\mathcal{F})$ is a {\it $\zz^d$ subshift of finite type} or {\it SFT}. For $d=1$, if $\mathcal{F} \subset \Sigma^{\{0,\ldots,k\}}$  we say that $X(\mathcal{F})$ is a {\it $k$-step SFT}.

For any $\zz^d$ subshift $X$, we define the {\it language} of $X$, written $\LL(X)$, to be the set of all configurations which appear
within points of $X$. For any finite $S \subset \zz^d$, denote by $\LL(X,S) := \LL(X) \cap \Sigma^S$ the set of configurations
in $\LL(X)$ with shape $S$. For any $\zz^d$ subshift $X$ and $w \in \LL(X,S)$, the set $[w] := \{x \in X \ : \ x_S = w\}$ is called
the {\it cylinder set} of $w$.

We now recall the definition of axial powers of subshifts as in \cite{LMP}, as already outlined in the introduction.
Given a $\zz$ subshift $X\subset \Sigma^{\mathbb{Z}}$ and any $d \in \mathbb{N}$, let $X^{\otimes d} \subset
\Sigma^{\mathbb{Z}^{d}}$ denote the $\mathbb{Z}^{d}$ subshift defined
by
$$X^{\otimes d} = \{ x \in \Sigma^{\mathbb{Z}^{d}} ~:~ \forall g \in \mathbb{Z}^{d} \ \forall i \in \{1,\ldots,d\}, \ x_{g + \mathbb{Z} e_i} \in X\},$$
where $x_{g + \mathbb{Z} e_i} \in \Sigma^\mathbb{Z}$ is the sequence obtained by shifting $x$ by $g$ and projecting it along the $i$th direction. $X^{\otimes d}$ is called the $d$th axial power of $X$.

Similarly, we define $\Xinf \subset \Sigma^{\mathbb{Z}^{\infty}}$ by
$$X^{\otimes \infty} = \{ x \in \Sigma^{\mathbb{Z}^{\infty}} ~:~ \forall g \in \mathbb{Z}^{\infty} \ \forall i \in \mathbb{N}, \ x_{g + \mathbb{Z} e_i} \in X\}.$$

It is clear that $X^{\otimes d}$ is always a $\mathbb{Z}^{d}$ subshift, since it is closed w.r.t. the product topology on
$\Sigma^{\zz^{d}}$, and is invariant with respect to the shift $\mathbb{Z}^{d}$-action on $\Sigma^{\mathbb{Z}^{d}}$.


The \textit{topological entropy} of a $\mathbb{Z}^d$ subshift $X$ ($d \in \mathbb{N} \cup \{\infty\}$) is defined as
$$h(X) = \lim_{n\to \infty}\frac{1}{|F_n|}\log |\mathcal{L}(X,F_n)|.$$

We define the \textit{limiting entropy} of a $\mathbb{Z}$ subshift $X$ as
\[
h_{\infty}(X) := \lim_{d \rightarrow \infty} h(X^{\otimes d}).
\]

The limit exists because $h(X^{\otimes d})$ is nonincreasing in $d$; see \cite{LMP} for a proof.

The next few definitions involve measures on $\mathbb{Z}^d$ subshifts ($d \in \mathbb{N} \cup \{\infty\}$). All such measures will be taken to be Borel probability measures with respect to the product topology. We say that $\mu$ is \textit{shift-invariant} if $\sigma_v \mu = \mu$ for all $v \in \zz^d$, 

For a shift-invariant measure $\mu$ on
$X$, the \textit{measure-theoretic entropy} of $\mu$ is
$$h(\mu):= \lim_{k \to \infty}\frac{1}{|F_n|}H_{\mu}(F_n),$$
where
$$H_{\mu}(F_n):=
-\int \log\mu([x_{F_n}])d\mu(x)$$
 is the
Shannon entropy of the discrete random variable $x_{F_n}$.

The above definitions and statements about entropy can be defined for any amenable group and are all independent of the particular choice of F\o lner sequence; for more information about entropy in this general setting, see \cite{ollagnier85}.

The following fundamental theorem of entropy theory generalizes to any countable amenable group acting by homeomorphisms on a compact space, and beyond.
For simplicity, we state it only for $\mathbb{Z}^d$ subshifts, where $d \in \mathbb{N} \cup \{\infty\}$, which is sufficient for the purposes of this paper:

\begin{theorem*}[\textbf{Variational principle
}]
For any $d \in \mathbb{N}\cup\{\infty\}$ and $\zz^d$ subshift $X$,
$$h(X) = \sup_\mu h(\mu),$$
where the supremum on the right-hand side is over all
shift-invariant measures on $X$.
\end{theorem*}

For a proof of the variational principle in the general amenable group setting, see \cite{ollagnier85}. For an elegant $\mathbb{Z}^d$-proof ($d <\infty$), see \cite{Mi}.

Since the shift $\zz^d$-action on the $d$-dimensional full shift
is always \emph{expansive}, the function $\mu \mapsto h(\mu)$ is \emph{upper semi-continuous} for measures on
any $\mathbb{Z}^d$ subshift. It follows  that the
supremum on the right hand side is attained by at least one
measure $\mu$. For proofs and details see \cite{bowen_entropy_expansive}, for instance.

A measure $\mu$ for which $h(X) =  h(\mu)$ is called a \emph{measure of maximal entropy} on $X$.

We also need the notion of conditional entropy. For any finite $F \subset \zz^d$, $d \in \mathbb{N} \cup \{\infty\}$,
any measure $\mu$ on the $d$-dimensional full shift, and any $\sigma$-algebra $\mathcal{A}$ which is measurable with
respect to $\mu$, we define
\[
H_{\mu}(F \ | \ \mathcal{A}) := - \int \log \mu([x_F] \ | \ \mathcal{A}) \ d\mu(x),
\]
and correspondingly define the {\it conditional entropy of $\mu$ with respect to $\mathcal{A}$} as
\[
h(\mu \ | \ \mathcal{A}) := \lim_{k \rightarrow \infty} \frac{1}{|F_n|} H_{\mu}(F_n \ | \ \mathcal{A}).
\]

A standard application of the definition of conditional entropy shows that we can also write
\[
h(\mu \ | \ \mathcal{A}) = \lim_{k \rightarrow \infty} \int \frac{1}{|F_n|} H_{\mu\mid\mathcal{A}}(F_n) \ d\mu(x),
\]
where $\mu \mid \mathcal{A}$ denotes the conditional measure of $\mu$ given $\mathcal{A}$.

Given two $\mathbb{Z}^d$ subshifts $X \subset \Sigma_1^{\mathbb{Z}^d}$ and $Y \subset \Sigma_2^{\mathbb{Z}^d}$, a (topological) \emph{factor map} is a surjective continuous map $\pi:X \to Y$ for which $\sigma_v \circ \pi = \pi \circ \sigma_v$ for all $v \in \mathbb{Z}^d$. Similarly, when $\mu$ and $\nu$ are shift-invariant measures on $X$ and $Y$ respectively, a \emph{measure-theoretic factor map} from $(X,\mu)$ to $(Y,\nu)$ is a map $\pi:X \to Y$ which is $\mu$-measurable, which projects the measure $\mu$ to $\nu$, and for which $\sigma_v \circ \pi = \pi \circ \sigma_v$ $\mu$-a.e. for all $v \in \mathbb{Z}^d$.
We will most often use conditional entropy in the case where $\mathcal{A}$ is the image of the underlying $\sigma$-algebra for the measure space under a $\mu$-measurable factor map $\pi$; in this case, we will write the factor $\pi$ in lieu of $\mathcal{A}$.

The following two well-known results relate non-conditional and conditional entropies.

\begin{theorem*}\label{abramov}{\bf (Rokhlin-Abramov formula \cite{ward_amenable_abramov_formula})}
For any shift-invariant measure $\mu$ on a $\zz^d$ subshift $X$ {\rm(}$d \in \mathbb{N} \cup \{\infty\}${\rm)} and any $\mu$-measurable factor map $\pi$ on $X$,
\[
h(\mu \ | \ \pi) = h(\mu) - h(\pi(\mu)).
\]
\end{theorem*}

\begin{theorem*}\label{past}{\rm (\bf Entropy via lexicographic past, see, for instance, p. 318 of \cite{georgii_book})}
For any $d \in \mathbb{N} \cup \{\infty\}$, denote by $P_d$ the lexicographic past of $0$ in $\zz^d$,
i.e. the set of all $g \in \zz^d$ which have at least one nonzero coordinate, the first of which is negative.
Then for any shift-invariant measure $\mu$ on $\Sigma^{\zz^d}$, 
$h(\mu) = H_{\mu}(\{0\} \ | \ \pi_{P_d})$. 
\end{theorem*}
\noindent
(Here and elsewhere, for a set $S \subseteq \zz^d$, $\pi_S$ represents the projection map $x \mapsto x_S$.)

In the case $d=1$ this is a classical result. The proof of the above result in \cite{georgii_book} formally applies only to the case $d < \infty$,
but is easily extended to the $d = \infty$ case.
Though we will not recreate the proof here, the idea is to think of an arbitrary cylinder set
on a large configuration as an intersection of events, and then to decompose the probability of
such an intersection as a product of conditional probabilities (using Bayes's rule) and apply the Martingale convergence theorem.


We will naturally identify $\zz^{d_1}$ as a subgroup of $\zz^{d_2}$
whenever $d_1 < d_2 \le \infty$ via the embedding $(v_1,\ldots,v_{d_1}) \mapsto (v_1,\ldots,v_{d_1},0,0,\ldots )$, and say that a sequence of measures
$\mu_n \in \mathbb{P}(\Sigma^{\zz^n})$ converges in the weak-$*$ topology
to a measure $\mu$ in $\mathbb{P}(\Sigma^{\zz^{\infty}})$ if
$\mu_n([w_F]) \to \mu([w_F])$ for all finite $F \subset
\zz^{\infty}$ and all $w_F \in \Sigma^F$.

We complete the section with two lemmas demonstrating connections between $\Xinf$ and the problem of limiting entropy, for which
we first need a preliminary result about legal configurations in $\Xinf$.

\begin{lemma}\label{samelang}
For any $\zz$ subshift $X$, any $d \in \mathbb{N}$, and any finite $S \subseteq \zz^d$, $\LL(X^{\otimes d},S) = \LL(X^{\otimes \infty},S)$.
\end{lemma}

\begin{proof}
The inclusion $\LL(X^{\otimes \infty},S) \subseteq \LL(X^{\otimes d},S)$ is trivial since the projection $x_{\mathbb{Z}^d}$ is clearly in $X^{\otimes d}$ for any $x \in X^{\otimes \infty}$. It therefore suffices to show the reverse inclusion.

Consider any $w \in \LL(X^{\otimes d},S)$. By definition, there exists $x \in X^{\otimes d}$ so that $x_S = w$. We will use $x$ to construct a point of $X^{\otimes \infty}$. Define $x' \in \Sigma^{\zz^{\infty}}$ by $\displaystyle x'_g := x_{(g_1, g_2, \ldots, g_{d-1}, \sum_{i=d}^{\infty} g_i)}$. Clearly $x'_{\zz^d} = x$, so $x'_S = w$. Also, for any $g \in \zz^{\infty}$ and $i \in \mathbb{N}$, the row $x'_{g + \zz e_i}$ is a row of $x$; either in the $x_i$-direction if $i < d$ or in the $x_d$-direction if $i \geq d$. Since $x \in X^{\otimes d}$, all such rows are in $X$, and so $x' \in X^{\otimes \infty}$. We have then shown that $w \in \LL(X^{\otimes \infty},S)$, completing the proof.

\end{proof}


\begin{lemma}\label{lem:hinf_hXinf}
$$h(\Xinf)=h_{\infty}(X).$$
\end{lemma}
\begin{proof}
By definition, we have
$$h_\infty(X)=\lim_{d \to \infty}h(X^{\otimes d}).$$
We have
$$h(X^{\otimes d})=\lim_{N \to \infty}\frac{1}{|F_N^{(d)}|} \log |\mathcal{L}(X^{\otimes d},F_N^{(d)})|.$$
Thus, we can find an increasing sequence of integers $(N_1,N_2,\ldots)$, so that
$$h_{\infty}(X) = \lim_{d \to \infty}\frac{1}{|F_{N_d}^{(d)}|} \log |\mathcal{L}(X^{\otimes d},F_{N_d}^{(d)})|.$$
As $F^{(d)}_{N_d}$ is a F{\o}lner sequence in $\zz^{\infty}$, we also have
$$h(\Xinf) =  \lim_{d \to \infty}\frac{1}{|F_{N_d}^{(d)}|} \log |\mathcal{L}(\Xinf,F_{N_d}^{(d)})|.$$
By Lemma~\ref{samelang} above, $\mathcal{L}(\Xinf,F_{N_d}^{(d)}) =  \mathcal{L}(X^{\otimes
d},F_{N_d}^{(d)})$; 
it follows that $h_{\infty}(X) = h(\Xinf)$.
\end{proof}
\begin{lemma}\label{lem:weak_limit_Xinf}
Let $\mu_1,\mu_2,\ldots, \mu_n,\ldots$ be a sequence of measures
where each $\mu_n$ is a measure on $X^{\otimes n}$.
\begin{enumerate}
  \item{If 
$\mu_n \rightarrow \mu$ in the weak-$*$ topology, where $\mu$ is a measure on $\Sigma^{\zz^{\infty}}$,
   then
  $\mu(\Xinf)=1$.}
  \item{If, in addition, every $\mu_n$ is a measure of maximal entropy of $X^{\otimes
  n}$, then 
$\mu$ is a measure of maximal
  entropy of $\Xinf$.}
\end{enumerate}
\end{lemma}
\begin{proof}

To prove (1), choose any sequence $\mu_n$ which approaches a weak-$*$ limit $\mu$. Choose any finite $F \subset \zz^{\infty}$. For some sufficiently large $d_0$, $F \subset \zz^{d_0}$, by definition of $\zz^\infty$.
Again, by Lemma~\ref{samelang}, $\mathcal{L}(X^{\otimes d_0},F)=\mathcal{L}(X^{\otimes d},F)$
for all $d_0 \le d \le \infty$.
Thus, for all $d \ge d_0$, $\mu_d(x_F \in \mathcal{L}(X^{\otimes d},F))=1$.
It follows that
$\mu(\{x_F \in \mathcal{L}(X^{\otimes \infty},F)\})=1$. Thus,
$$ \mu(X^{\otimes \infty})=
\mu\left(\bigcap_{F \subset \zz^\infty}\{ x_F \in \mathcal{L}(X^{\otimes \infty},F)\} \right)=1,$$
where the intersection is over all finite $F \subset \zz^\infty$. This completes the proof of (1).

Now, (2) follows by combining (1), the relation $h(\Xinf)=h_{\infty}(X)$ from Lemma~\ref{lem:hinf_hXinf}, and upper semi-continuity of measure-theoretic entropy.
\end{proof}

\subsection{Multi-choice subshifts and independence entropy}

We now recall some definitions and results from \cite{LMP}. Let $\hat{\Sigma}$ denote the set of non-empty subsets of
$\Sigma$. Let $X$ be a $\zz^d$ subshift over $\Sigma$, where $d
\in \mathbb{N}\cup \{\infty\}$. The \emph{multi-choice shift}
$\widehat{X} \subset \hat{\Sigma}^{\zz^d}$ is the set
$$ \widehat{X} =  \{ \hat{x} \in \hat{\Sigma}^{\zz^d}~:~ \hat{x} \subset X\},$$
where $\hat{x} \in \hat{\Sigma}^{\zz^d}$ is naturally interpreted as
a subset of $X$, obtained by Cartesian products.
Note that any $\hat{x} \in \widehat{X}$ corresponds to an (infinite) ``independently legal'' point as described in the introduction.

The \emph{independence score} of a configuration $\hat{w} \in \hat{\Sigma}^F$ with shape $F$ is defined by
$$S(\hat{w}) = \frac{1}{|F|}\sum_{n \in F}\log |\hat{x}_n|.$$
We analogously define the independence score for $\hat{x} \in \hat{\Sigma}^{\zz^d}$ by
$$S(\hat{x})= \limsup_{n \to \infty}  S(\hat{x}_{F_n}).$$ 
Finally, we define the independence score of any
shift-invariant measure $\hat{\mu}$ on $\widehat{X}$ as
$$S(\hat{\mu}) = \int \log |\hat{x}_0| \ d\hat{\mu}(\hat{x}).$$

Observe that for any shift-invariant measure $\hat{\mu}$ and finite $F
\subset \zz^{d}$, $S(\hat{\mu})=\int S(\hat{x}_F) d\hat{\mu}(\hat{x})$. In
particular, since $\hat{x} \to S(\hat{x})$ is a function which
is invariant under shifts, it follows that for an ergodic measure
$\hat{\mu}$, $S(\hat{x})=S(\hat{\mu})$ $\hat{\mu}$-almost everywhere. Also, by the
pointwise ergodic theorem, the $\limsup$ in the definition of $S(\hat{x})$ is actually a
limit $\hat{\mu}$-almost surely.

Following \cite{LMP}, we define the \emph{independence entropy} $\hind(X)$ of a $\zz$ subshift $X$ as
$$\hind(X)=\lim_{n \to \infty}\left(\sup\{S(\hat{w})~:~  \hat{w} \in \mathcal{L}(\hat{X},F_n)\}\right).$$
See Section $4$ of \cite{LMP} for details on how existence of the
limit follows from sub-additivity. The above definition of $\hind(X)$ is a specific instance of a \emph{maximal ergodic average}, in this case on the compact space $\hat{X}$. The study of maximal ergodic averages goes by the name \emph{ergodic optimization}. The paper \cite{jenk} by O. Jenkinson establishes basic aspects of the theory.

The following lemma states convenient equivalent definitions of the
independence entropy:

\begin{lemma}\label{lem:hind_equiv}
Let $X$ be a $\mathbb{Z}^d$ subshift. The following are equivalent
definitions of the independence entropy $\hind$:
\begin{enumerate}

\item{$\hind(X)=\sup\{S(\hat{x})~:~ \hat{x} \in \hat{X}\}$.}
\item{$\hind(X)$ is equal to $\sup_{\hat{\mu}} S(\hat{\mu})$, where the supremum is over shift-invariant measures on $\hat{X}$.}

\end{enumerate}
\end{lemma}

Lemma \ref{lem:hind_equiv} above is a special case of Proposition 2.1 from \cite{jenk}, which shows that several different ways of taking maximal ergodic averages for a continuous function $f$ on a compact metric space $K$ under a continuous map $T : K \rightarrow K$ are equal. Our $K$ is $\hat{X}$, our function $f$ is given by $\hat{x} \mapsto \log |\hat{x}_0|$, and our $T$ is the shift map. Three of the ``averages'' that he defines are $\gamma(f)$ (which is $\sup_{\hat{x}} S(\hat{x})$ in our context), $\alpha(f)$ (which is $\sup_{\hat{\mu}} S(\hat{\mu})$ in our context), and $\delta(f)$ (which is our definition of independence entropy), and Proposition 2.1 from \cite{jenk} states that the three are equal. The proof there is technically written only for $\mathbb{Z}$-actions, but it clearly extends to the $\mathbb{Z}^d$ case.
For the convenience of the reader, we present a short self-contained specific proof, essentially following \cite{jenk}.

\begin{proof}

 It should be clear from the definitions that $\hind(X) \geq \sup_{\hat{x}} S(\hat{x})$. Also, since $\hat{x} \mapsto \log |\hat{x}_0|$ is an $L^1$ function, for every $\hat{\mu}$, by the pointwise ergodic theorem, $S(\hat{\mu}) = \int S(\hat{x}) \ d\hat{\mu}(\hat{x})$, and so there exists $\hat{x} \in \hat{X}$ s.t. $S(\hat{x}) \geq S(\hat{\mu})$. Therefore, $\sup_{\hat{x}} S(\hat{x}) \geq \sup_{\hat{\mu}} S(\hat{\mu})$. It remains to show that $\sup_{\hat{\mu}} S(\hat{\mu}) \geq \hind(X)$. Take configurations $\hat{w}^{(n)} \in \mathcal{L}(\hat{X},F_n)$ s.t. $S(\hat{w}^{(n)}) \rightarrow \hind(X)$. Since each $\hat{w}^{(n)} \in \mathcal{L}(\hat{X})$, we can define a sequence of points $\hat{x}^{(n)} \in \hat{X}$ with $(\hat{x}^{(n)})_{F_n} = \hat{w}^{(n)}$. Then, define measures $\hat{\mu}^{(n)} = \frac{1}{|F_n|} \sum_{v \in F_n} \delta_{\sigma_v \hat{x}^{(n)}}$. By compactness of $\hat{X}$, we may take a weak-* limit point $\hat{\mu}$ of the sequence $(\hat{\mu}^{(n)})$, which will be shift-invariant. Finally,
$S(\hat{\mu}) = $
\[
\lim_{n \rightarrow \infty} \int \log |\hat{x}_0| \ d\hat{\mu}^{(n)}(\hat{x}) = \lim_{n \rightarrow \infty} \frac{1}{F_n} \sum_{v \in F_n} \log |(\hat{x}^{(n)})_v| = \lim_{n \rightarrow \infty} S(\hat{w}^{(n)}) = \hind(X).
\]

\end{proof}

\subsection{Exchangeability and de Finetti's theorem}
One of the main tools in the proof of Theorem~\ref{hind=hinf} is de
Finetti's Theorem, which we review here for completeness.

Let $\mathbb{P}$ 
denote the group of
finite permutations of the positive integers, i.e. permutations of
$\mathbb{N}$ which fix all but finitely many integers. $\mathbb{P}$
is a countable, amenable group.

\begin{definition}
A sequence $(X_n)$ of random variables is called
\textbf{exchangeable} if for any $n$ and n-tuples of distinct
integers $i_1, i_2, \ldots, i_n$ and $j_1, j_2, \ldots, j_n$, the
joint distributions of $X_{i_1}, \ldots, X_{i_n}$ and $X_{j_1},
\ldots, X_{j_n}$ are the same. Equivalently, $(X_n)$ is exchangeable
if all joint distributions are invariant under the action of $\mathbb{P}$
on $(X_n)$ by permutation of indices.
\end{definition}

The simplest examples of exchangeable sequences of random variables
are i.i.d. or Bernoulli sequences,  which are clearly exchangeable.
However, not every exchangeable sequence is Bernoulli. For instance,
one can define the $X_i$ so that either they must all be equal to
$0$ (say with probability $0.3$) or all must be equal to $1$ (say
with probability $0.7$). The reader may check that this sequence is
exchangeable, but clearly it is not Bernoulli. However, it is a
mixture of the two (trivial) Bernoulli sequences which are a.s.
constant with value $0$ and a.s. constant with value $1$ respectively.
In fact, this is not an anomaly: de Finetti's theorem
states that all exchangeable sequences are mixtures of i.i.d. distributions.

\begin{theorem*}[\textbf{de Finetti's theorem}]
\label{definetti} Any exchangeable sequence $(Y_n)$ of random
variables each taking values in a  Borel space $(\Omega,\mathcal{F})$ is a mixture of
identically distributed random variables. In other words, there is a
measure $\theta$ on the simplex of Borel probability measures on $(\Omega,\mathcal{F})$ such that

$$P\left(\bigcap_{k=1}^N\{Y_k \in A_k \}\right)
=\int \prod_{k=1}^N p(A_k) d\theta(p),$$
for any $A_1,\ldots A_N \in \mathcal{F}$.
\end{theorem*}
This version of de Finetti's theorem
for random variables taking values in a Borel measurable space is due to Hewitt and Savage \cite{hewittsavage}.

In particular, for exchangeable random variables, the
\emph{exchangeable} $\sigma$-algebra coincides with the \emph{tail}
$\sigma$-algebra: Any measurable function of an exchangeable
sequence which is invariant under finite permutations of the
variables is measurable with respect to the tail. In the case of
exchangeable variables taking values in a finite set $\Omega$, the
particular i.i.d. distribution in the mixture can be recovered by observing the
\emph{empirical distributions}
$\overline{\mu}_a(X_1,\ldots,X_n,\ldots):=\lim_{n \to
\infty}\frac{1}{n}|\{i\le n ~:~ X_i = a\}|$, $a \in \Omega$.


\subsection{Allowable local perturbations}
A useful notion in the proof and statement of our results is the
notion of ``allowable local perturbations'' for a point in a $\mathbb{Z}^d$ subshift.




For any subshift $X \subset \Sigma^{\zz^d}$, $d \in \mathbb{N} \cup \{\infty\}$, shift-invariant
measure $\mu$ on $X$, $x \in X$, and $g \in \zz^d$, we say that a letter $a$ is a
{\it $\mu$-allowable local perturbation of $x$ at $g$} if the conditional probability
$$p_{x,g}(a):=\mu(\{z \in X:z_g = a\} \mid
\{ z \in X ~:~ z_{\{g\}^c}=x_{\{g\}^c}\})$$
is greater than $0$.

For any shift-invariant measure $\mu$ on $X$, define a 
map $\phi_\mu:X \to \hat{\Sigma}^{\zz^d}$ by
$$ (\phi_{\mu}(x))_g = \{a \in \Sigma ~:~ p_{x,g}(a) >0 \}.$$
This is a measurable factor map to $\hat{\Sigma}^{\zz^d}$; $\mu$-almost surely, $(\phi_\mu(x))_g \ne
\varnothing$ for all $g \in \zz^d$, since by definition $\mu(\{x_g \in (\phi_{\mu}(x))_g\})=1$.

\section{Proof of theorem \ref{hind=hinf}}\label{mainproof}

The group $\mathbb{P}$ acts on $\mathbb{Z}^{\infty}$ in a natural way by permuting coordinates:
for $\rho \in \mathbb{P}$ and $g \in \mathbb{Z}^{\infty}$, $(\rho(g))_i := g_{\rho(i)}$.
Through the action of $\mathbb{P}$ on $\mathbb{Z}^{\infty}$,
$\mathbb{P}$ also acts on $X^{\otimes \infty}$.

We will consider measures of maximal entropy on $X^{\otimes \infty}$
which are in addition invariant with respect to the action of
$\mathbb{P}$ on $X^{\otimes \infty}$. The existence of such measures follows
from amenability of $\mathbb{P}$, upper semi-continuity of 
measure-theoretic entropy, 
and the fact that the action
of $\mathbb{P}$ on $\Xinf$ preserves measure-theoretic entropy.

To be specific, choose any shift-invariant measure
$\nu$ on $X^{\otimes \infty}$, and take any weak-$*$ limit point $\mu$
of the sequence $\displaystyle \frac{1}{|P_n|} \sum_{\rho \in P_n} \rho \circ
\nu$, where
$P_n \subset \mathbb{P}$ is the set of permutations which fix all
integers greater than $n$.

Clearly, $h(\nu)=h(\rho\nu)$ for any permutation $\rho
\in \mathbb{P}$. Since $\nu \mapsto h(\nu)$ is an
affine and upper semi-continuous function, it follows that $\mu$ is
a measure of maximal entropy whenever $\nu$ is a measure of maximal
entropy.

\begin{lemma}\label{lem:cord_ind}
For any $\mathbb{P}$-invariant measure $\nu$ on $X^{\otimes \infty}$
and any finite $F \subset \mathbb{Z}^{\infty}\setminus \{0\}$, $x_0$
and $x_F$ are $\nu$-almost surely $\nu$-conditionally independent given $x_{F^c\setminus
\{0\}}$.
\end{lemma}

\begin{proof}
For $n \in \mathbb{Z}^\infty$, we define its {\it support} by $supp(n)=\{ k \in \mathbb{N} ~:~ n_k \ne 0\}$.
Let $F \subset \mathbb{Z}^{\infty}\setminus \{0\}$ be any finite set.
We denote by $S_{d}(F)$ the set of all $n \in \mathbb{Z}^d \setminus (F\cup\{0\})$ for which there is some $k \in \mathbb{N}$ such that $supp(n) \subset ((d-1)k,d\cdot k]$.

Choose any $d_0 \in \mathbb{N}$ so that $F \subset \mathbb{Z}^{d_0}$ (naturally identifying $\mathbb{Z}^{d_0}$ as a subset of $\mathbb{Z}^{\infty}$), equivalently
$\bigcup_{n \in F}supp(n) \subset [1,d_0]$. For $k \in \mathbb{N}$, define $\tau_k : \mathbb{Z}^\infty \to \mathbb{Z}^{d_0}$ by
$\tau_k(n)_{j} = n_{j+d_0 k}$ for $1 \le j \le d_0$. 

Then, define $\xi_k:X^{\otimes \infty} \to \Sigma^{\mathbb{Z}^{d_0} \setminus \{0\}}$ which projects onto the set of nonzero $n \in \mathbb{Z}^{\infty}$ whose support is contained in $(d_0(k-1),d_0\cdot k]$, which are mapped to $\mathbb{Z}^{d_0} \setminus \{0\}$ via $\tau_k$:

$$(\xi_k(x))_n := x_{\tau_k(n)} \mbox{ for } n \in \mathbb{Z}^{d_0} \setminus \{0\}.$$

We consider $\xi_1,\xi_2,\ldots,\xi_k,\ldots$ as a sequence of random variables, taking values in $\Sigma^{\mathbb{Z}^{d_0} \setminus \{0\}}$, with respect to the measure $\nu$ on $X^{\otimes \infty}$.
It follows from $\mathcal{P}$-invariance of $\nu$ that the sequence of random variables $\{\xi_k\}_{k=1}^\infty$ is exchangeable.
Thus, by de Finetti's theorem their joint distribution is a mixture of independent identically distributed random variables.
Furthermore, the joint distribution of $x_0$ and the $\{\xi_k\}$'s is invariant under permutations of the sequence $\{\xi_k\}_{k=1}^\infty$. It follows that given $\{\xi_k\}_{k \in S}$ for any infinite $S \subset \mathbb{N}$, $x_0$ is independent of $\{\xi_k\}_{k \in \mathbb{N} \setminus S}$.

We have concluded that $\nu$-almost surely, $x_0$ and $\xi_1 = \{ x_n ~:~ n \in \mathbb{Z}^{d_0} \setminus \{0\} \}$ are conditionally independent given $\{\xi_k\}_{k > 1} = \{ x_n ~:~ n \in S_{d_0}(F) \setminus \mathbb{Z}^{d_0}\}$. The reader may check that additionally conditioning on some subset of $\{ x_n ~:~ n \in \mathbb{Z}^{d_0} \setminus \{0\} \}$ (along with $\{ x_n ~:~ n \in S_{d_0}(F) \setminus \mathbb{Z}^{d_0}\}$) preserves conditional independence of $x_0$ and the remaining sites in $\{ x_n ~:~ n \in \mathbb{Z}^{d_0} \setminus \{0\} \}$. In particular, this means that $\nu$-almost-surely, $x_0$ and $x_F$ are conditionally independent given $\{x_n ~:~ n \in S_{d_0}(F)\}$.

Observe that this holds for all $d > d_0$ as well. Also note that $S_{d_1}(F) \subset S_{d_2}(F)$ whenever $d_1$ divides $d_2$, and so $\{S_{d!}(F)\}_{d \geq d_0}$ is an increasing sequence of subsets of $\mathbb{Z}^\infty$ whose union is $\mathbb{Z}^\infty \setminus F \cup\{0\}$. Thus, $\nu$-almost-surely $x_0$ and $x_F$ are independent given $x_{F^c\setminus\{0\}}$, completing the proof.

\end{proof}


\begin{lemma}\label{indepupdates}
Let $\mu$ be a measure on $X^{\otimes \infty}$ which is both shift-invariant and $\mathbb{P}$-invariant. Then for a set of $x \in X^{\otimes \infty}$ of full $\mu$-measure, any $y \in \Sigma^{\mathbb{Z}^{\infty}}$ obtained by a finite number of $\mu$-allowable local-perturbations of $x$ is in $X^{\otimes \infty}$. Also, for any such $x$, $\phi_\mu(x) \in \widehat{X^{\otimes \infty}}$.
\end{lemma}

\begin{proof}
By compactness of $X^{\otimes \infty}$, it is sufficient to show that for $\mu$-a.e. $x \in X^{\otimes \infty}$,
any finite $F \subset \mathbb{Z}^{\infty}$, and any configuration $y_F \in \Sigma^F$ with $y_g \in (\phi_{\mu}(x))_g$ for all $g \in F$,
there exists $z \in X^{\otimes \infty}$ with $z_F=y_F$.

It will suffice to show that for $\mu$-a.e. $x \in X^{\otimes \infty}$ and for any choices $y_g \in (\phi_{\mu}(x))_g$ for all $g \in F$,
\begin{equation}
\label{eq:ind_change_cord}
 \mu( \{z \in X^{\otimes \infty}: z_F = y_F \mid z_{F^c}=x_{F^c}\}) =
\prod_{ g \in F} p_{x,g}(y_g).
\end{equation}
This is because integrating \eqref{eq:ind_change_cord} with respect to $\mu$ shows that $\mu([y_F])>0$,
and since $\mu$ is a measure on $X^{\otimes \infty}$, this would show that $y_F \in \LL(X^{\otimes \infty})$,
completing the proof as explained above.

We prove (\ref{eq:ind_change_cord}) by induction on $n = |F|$. For $n=1$, (\ref{eq:ind_change_cord}) is just the definition of $p_{x,g}$.
For the inductive step, assume that $|F| = n$ and that (\ref{eq:ind_change_cord}) is true for all sets of
cardinality less than $n$.

Choose some $g \in F$, and let $F_1 := F \setminus \{g\}$. For the purposes of the proof, we extend $y_F$ to a full point $y \in \Sigma^{\zz^{\infty}}$ which is equal to $y_F$ on $F$ and equal to $x_{F^c}$ on $F_c$. For $\mu$-a.e. $x \in X^{\otimes \infty}$, the definition of conditional probability gives
\begin{multline*}
\mu( \{z \in X^{\otimes \infty}: z_F = y_F \mid z_{F^c}=x_{F^c}\}) \\
= \mu( \{z \in X^{\otimes \infty}: z_{F_1} = y_{F_1} \mid z_{F_1^c}=y_{F_1^c}\}) \mu(\{z_g = y_g \mid z_{F^c} = x_{F^c}\}).
\end{multline*}

The inductive hypothesis on $F_1$ implies that
$$\mu( \{z \in X^{\otimes \infty}: z_{F_1} = y_{F_1} \mid z_{F_1^c}=y_{F_1^c}\})=\prod_{h \in F_1}p_{y,h}(y_h),$$
and by Lemma~\ref{lem:cord_ind}, $z_{F_1}$ and $z_g$ are conditionally $\mu$-independent given $z_{F^c}$, and so $\mu(\{z_g = y_g \mid z_{F^c} = x_{F^c}\}) = \mu(\{z_g = y_g \mid z_{\{g\}^c} = x_{\{g\}^c}\}) = p_{x,g}(y_g)$. Therefore, proving (\ref{eq:ind_change_cord}) is reduced to showing that $p_{y,h}(y_h)=p_{x,h}(y_h)$ for all $h \in F_1$.

Again by Lemma \ref{lem:cord_ind}, for every $h \in F_1$, $z_{F \setminus \{h\}}$ and $z_h$ are conditionally independent given $z_{F^c}$, and so
$$p_{y,h}(y_h) = \mu(\{z \in X^{\otimes \infty} : z_h=y_h \mid z_{\{h\}^c}=y_{\{h\}^c}\})$$

$$\textrm{and } p_{x,h}(y_h) = \mu(\{z \in X^{\otimes \infty} : z_h=y_h \mid z_{\{h\}^c}=x_{\{h\}^c}\})$$
are both equal to $\mu(\{z \in X^{\otimes \infty} : z_h=y_h \mid z_{F^c}=x_{F^c}\})$.

In particular, this shows that $p_{y,h}(y_h)=p_{x,h}(y_h)$ for all $h \in F_1$, completing the proof.


\end{proof}

In other words, for $\mu$ as above, the measure-theoretic factor map $\phi_\mu$ maps to $\widehat{X^{\otimes \infty}}$, $\mu$-almost surely. 
We denote by $\hat{\mu}$ the measure $\phi_\mu(\mu)$ on $\widehat{\Xinf}$, which is just the pushforward of $\mu$ under $\phi_{\mu}$.

Our next step is to characterize the structure of $\mathbb{P}$-invariant measures of maximal entropy on $\Xinf$. 

\begin{lemma}\label{lem:mme_rel_ind}
For any $\mathbb{P}$-invariant measure of maximal entropy $\mu$ on $\Xinf$,
\begin{equation}
\label{eq_mu_over_hat}
\mu = \int_{\widehat{\Xinf}} \mu_{\hat{x}} \ d\hat{\mu}(\hat{x}),
\end{equation}
where for any $\hat{x} \in \widehat{\Xinf}$, $\mu_{\hat{x}}$ denotes the independent product of the uniform measures on $\hat{x}_g$ over all $g \in \mathbb{Z}^{\infty}$.
\end{lemma}
\begin{proof}
The conditional entropy of $\mu$ given $\phi_\mu$ satisfies
$$ h(\mu \mid \phi_\mu) \ = \int_{\widehat{\Xinf}}H_\mu( \{0\} \mid \phi_{\mu}(x)) \ d\hat{\mu} \le \int_{\widehat{\Xinf}}\log|\hat{x}_0| \ d\hat{\mu}(\hat{x})$$
The equality follows since by Lemma~\ref{indepupdates}, all sites $x_v$ are conditionally independent given $\phi_{\mu}(x)$, and the inequality comes from the trivial fact that $H_{\mu}(R) \leq \log k$ for any random variable $R$ taking on $k$ values. Clearly equality holds in the inequality iff the conditional distribution of $x_0$ under $\mu$ given $\phi_\mu(x)$ is $\mu$-a.s. uniform over the finite set $\phi_\mu(x)_0$.

Thus, if $\mu$ does not satisfy the formula \eqref{eq_mu_over_hat}, we could define a measure $\displaystyle \nu = \int_{\widehat{\Xinf}} \mu_{\hat{x}} \ d\hat{\mu}(\hat{x})$ for which $h(\nu) > h(\mu)$, which would contradict the fact that $\mu$ is a measure of maximal entropy.
\end{proof}

\begin{lemma}\label{zeroentfactor}

Let $\mu$ be a $\mathbb{P}$-invariant measure of maximal entropy on $\Xinf$.
The measure-theoretic entropy of the measure $\hat{\mu}$ with
respect to the $\mathbb{Z}^{\infty}$-action by shifts is zero:
$$h(\hat{\mu})=0.$$
\end{lemma}

\begin{proof}
Suppose for a contradiction that $h := h(\hat{\mu}) > 0$. First, we note that by definition, for any $g \in \mathbb{Z}^{\infty}$, $(\phi_{\mu}(x))_g$ is $\mu$-a.s. uniquely determined by $x_{\{g\}^c}$ as the set $\{a \in \Sigma \ : \ 
p_{x,g}(a) > 0\}$. But by Lemma~\ref{lem:mme_rel_ind}, $\mu$-a.s., $(\phi_{\mu}(x))_g$ is independent of $x_{\{g\}^c}$ given $(\phi_\mu(x))_{\{g\}^c}$. It is clear that if two random variables are independent and the first is a function of the second, then the first must be constant. Therefore, $(\phi_{\mu}(x))_g$ must be conditionally constant given $(\phi_{\mu}(x))_{\{g\}^c}$. In other words, $\mu$-a.s., $(\phi_{\mu}(x))_{\{g\}^c}$ uniquely determines $(\phi_{\mu}(x))_g$.

This means that there exists $N$ so that $(\phi_{\mu}(x))_0$ is determined by $(\phi_{\mu}(x))_{F_{N} \setminus \{0\}}$ with $\mu$-probability  $1 - \delta$, for some $\delta < \frac{h}{3 \log |\Sigma|}$. Choose $d$ large enough that $\frac{1}{|F_{d}|} H_{\hat{\mu}}(F_{d}) \le h + \epsilon$, where $\epsilon < \frac{h}{3 |F_N|}$. 
We now decompose $H_{\hat{\mu}}(F_N \times F_d)$ as a sum of conditional entropies:
\[
H_{\hat{\mu}}(F_N \times F_d) = H_{\hat{\mu}}((F_N \setminus \{0\}) \times F_d) + H_{\hat{\mu}}(\{0\} \times F_d \ | \ \pi_{(F_N \setminus \{0\}) \times F_d})
\]
\[
\leq (|F_N| - 1) |F_d| (h + \epsilon) + |F_d| \delta \log |\Sigma|
\]
\[
= |F_N| |F_d| \left[h - \frac{h}{|F_N|} + \epsilon - \frac{\epsilon}{|F_N|} + \frac{\delta \log |\Sigma|}{|F_N|}\right]
\leq |F_N| |F_d| \left[h - \frac{h}{|F_N|} + \epsilon + \frac{\delta \log |\Sigma|}{|F_N|} \right]
\]
\[
< |F_N| |F_d| \left[h - \frac{h}{|F_N|} + \frac{h}{3|F_N|} + \frac{h}{3|F_N|}\right] = |F_N| |F_d| h \left(1 - \frac{1}{3|F_N|}\right).
\]
From the general theory of entropy for amenable groups (as in
\cite{ollagnier85}, p. $59$), for any finite set $P\subset
\zz^\infty$, $\frac{1}{|P|} H_{\hat{\mu}}(P) \geq h$, and so we have
a contradiction.

\end{proof}

\begin{proof}[Proof of Theorem~\ref{hind=hinf}] Let $\mu$ be a $\mathbb{P}$-invariant measure of maximal entropy on $\Xinf$.
By Lemma~\ref{indepupdates}, $\phi_\mu$ is a measure-theoretic factor
from the $\zz^{\infty}$-system $(X^{\otimes \infty}, \mu)$ into
$(\widehat{X^{\otimes \infty}}, \hat{\mu})$.

The Rokhlin-Abramov formula gives
$$h(\mu)=h(\hat{\mu})+h(\mu \mid \phi_{\mu}),$$
where $$h(\mu \mid \phi_{\mu}) = \int_{\widehat{X^{\otimes \infty}}}
\lim_{n \rightarrow \infty} \frac{1}{|F_n|}H_\mu(F_n \mid
\phi_{\mu}(x)) \ d\mu$$ is the relative entropy of $\mu$ over
$\phi_{\mu}$. However, by Lemma~\ref{zeroentfactor}, $h(\hat{\mu}) = 0$.
By definition of $\phi_{\mu}$, $\frac{1}{|F_n|} H_\mu(F_n \mid \phi_{\mu}(x)) \le S((\phi_{\mu}(x))_{F_n})$, and so
\begin{equation}\label{hinfdecomp}
h_{\infty}(X) = h(\mu) =
h(\mu \mid \phi_{\mu}) \leq \int_{\widehat{X^{\otimes \infty}}} \limsup_{n \rightarrow \infty} S(\hat{x}_{F_n}) \ d\hat{\mu}(\hat{x}).
\end{equation}

However, for any $\hat{x} \in \widehat{X^{\otimes \infty}}$,
$$ \limsup_{n \rightarrow \infty} S(\hat{x}_{F_n}) \le \hind(X),$$
so $h_{\infty}(X) \leq \hind(X)$. Finally, we already know from \cite{LMP} that $\hind(X) \leq h_{\infty}(X)$, so we have
shown that $h_{\infty}(X) = \hind(X)$.

\end{proof}

As an application, we now verify the formula for the limiting entropy of the subshift $X_{ADD}$ given in the introduction, where $\Sigma$ is the set of letters in the English alphabet
and $ADD$ is the only forbidden word. Given Theorem~\ref{hind=hinf}, clearly it suffices to show that $\hind(X_{ADD}) = \frac{1}{2} \log 25 + \frac{1}{2} \log 26$. We sketch an argument here,
and leave it to the reader to fill in details.

The alphabet $\hat{A}$ of $\widehat{X_{ADD}}$ technically has $2^{26} - 1$ elements, but it is clear that if one wishes to maximize indepedence score, only four should be used.
Define $S_1 = \Sigma \setminus \{A,D\}$, $S_2 = \Sigma \setminus \{A\}$, $S_3 = \Sigma \setminus \{D\}$, and $S_4 = \Sigma$. Then in any legal configuration in $\widehat{X_{ADD}}$, one can find another
legal configuration, using only the $S_i$, with equal or greater independence score, by including all letters of $\Sigma$ except $A$ and $D$ at every site. To compute $\hind(X_{ADD})$, we can
then restrict our attention to configurations using only the $S_i$.

We first note that the word $S_4 S_4 S_4$ is illegal in $\widehat{X_{ADD}}$. Therefore, any word in $\LL(\widehat{X_{ADD}})$ can be written as
\[
w = w_0 S_4^{n_1} w_1 S_4^{n_2} \ldots w_{k-1} S_4^{n_k} w_k,
\]
where each $n_i$ is $1$ or $2$, every $w_i$ except $w_0$ is nonempty, and each $w_i$ contains no $S_4$s. We further note that if $S_4 S_4$ occurs in a legal configuration in $\widehat{X_{ADD}}$, then since $D \in S_4$, the preceding letter $S_i$ must not include $A$, and so must be $S_1$ or $S_2$. But if it is preceded by $S_2$, then since $D \in S_2$, again the preceding letter must be $S_1$ or $S_2$. Continuing in this way, we see that if any $n_i = 2$, then $w_{i-1}$ is either of the form $S_2^k$ or ends with $S_1 S_2^k$ for some $k$. Furthermore, the only $w_{i-1}$ which can possibly be $S_2^k$ is $w_0$, since all other $w_{i-1}$ are preceded by an $S_4$, and $S_4 S_2 S_2$ and $S_4 S_2 S_4$ are both illegal configurations for $\widehat{X_{ADD}}$.

Summarizing, we see that $w$ can be decomposed into (possibly empty) $w_0$ and $w_k$ which contain no $S_4$, the word $S_4^{n_1}$, words of the form $w_i S_4$ where $w_i$ is nonempty and has no $S_4$, and words of the form $w_i S_4 S_4$ where $w_i$ is nonempty, has no $S_4$, and has at least one $S_1$. The reader may check that except for $S_4^{n_1}$, all of the remaining types of words have independence score less than or equal to their length times $\frac{1}{2} \log 25 + \frac{1}{2} \log 26$. Therefore, any $w \in \LL(\widehat{X_{ADD}})$ has independence score less than or equal to $2\log 26 + |w| (\frac{1}{2} \log 25 + \frac{1}{2} \log 26)$, which implies that $\hind(X_{ADD}) \leq \frac{1}{2} \log 25 + \frac{1}{2} \log 26$.

Finally, clearly $\hind(X_{ADD}) \geq \frac{1}{2} \log 25 + \frac{1}{2} \log 26$, since all words $(S_3 S_4)^n$ are in $\LL(\widehat{X_{ADD}})$ and such words have independence score equal to their length times $\frac{1}{2} \log 25 + \frac{1}{2} \log 26$. Therefore, $\hind(X_{ADD}) = \frac{1}{2} \log 25 + \frac{1}{2} \log 26$.

\section{Limiting measures of maximal entropy}\label{mmes}

We now wish to discuss the structure of the measure(s) of maximal
entropy on $X^{\otimes \infty}$. We will show that we can
completely describe the structure of the measures of maximal entropy which are
$\mathbb{P}$-invariant. Denote by $\Xmax \subset \widehat{X^{\otimes
\infty}}$ the set of points $\hat{x}$ with $S(\hat{x})=
\hind(X)$. This is a shift-invariant and $\mathbb{P}$-invariant subset of
$\widehat{\Xinf}$.

For any shift-invariant measure $\nu$ supported on $\widehat{\Xinf}_{\max}$, define a shift-invariant measure
$\Phi(\nu)$ on $\Xinf$ by
\begin{equation}
\label{eq_Phi}
\Phi(\nu) = \int_{\widehat{X^{\otimes \infty}}} \mu_{\hat{x}} \ d\nu(\hat{x}),
\end{equation}
where for any $\hat{x} \in \widehat{\Xinf}_{\max}$, $\mu_{\hat{x}}$ is the independent product
 of the uniform measures over $\hat{x}_g$ for $g \in \mathbb{Z}^{\infty}$.

\begin{theorem}\label{mmestructure}
Any $\mathbb{P}$-invariant measure of maximal entropy $\mu$ on $X^{\otimes \infty}$ is of the form $\mu = \Phi(\hat{\mu})$,
where $\hat{\mu}$ is the push-forward of $\mu$ via $\phi_\mu$ and is  supported on $\Xmax$.
\end{theorem}

\begin{proof}

If $\mu$ is a $\mathbb{P}$-invariant
measure of maximal entropy for $X^{\otimes \infty}$, define $\hat{\mu} = \phi_{\mu}(\mu)$. By Lemma~\ref{lem:mme_rel_ind}, $\mu = \Phi(\hat{\mu})$. We also have the following chain of inequalities:
\[
h_{\infty}(X) = \ h_{\mu}(\{0\} \ | \ \pi_{P_{\infty}}) = \int_{X^{\otimes \infty}} H_{\mu}(\{0\} \ | \ \pi_{P_{\infty}}) \ d\mu(x)
= \int_{X^{\otimes \infty}} H_{\mu}(\{0\} \ | \ \pi_{P_{\infty}} \times \phi_{\mu}) \ d\mu(x)
\]
\[
\leq \int_{X^{\otimes \infty}} H_{\mu}(\{0\} \ | \ \phi_{\mu}) \ d\mu(x) \leq \int_{X^{\otimes \infty}} \log |(\phi_{\mu}(x))_0| \ d\mu(x) = \int_{\widehat{X^{\otimes \infty}}} S(\hat{x}_0) \ d\hat{\mu}(\hat{x})
\leq \hind(X).
\]

The third equality holds since $h({\hat{\mu}}) = 0$. By Theorem~\ref{hind=hinf}, $h_{\infty}(X) = \hind(X)$. Therefore, all inequalities above are in fact equalities. The first inequality being an equality implies that $x_0$ is conditionally $\mu$-independent from $x_P$, given $\phi_{\mu}(x)$. This clearly implies that $\mu$-almost every fiber $\mu_{\hat{x}}$ in the disintegration of $\mu$ over $\phi_{\mu}$ is a sitewise independent product. The second inequality being an equality implies that $\mu$-a.s., the distribution of $\mu_{\hat{x}}$ on a site $g \in \mathbb{Z}^{\infty}$ is uniform over $\hat{x}_g$. Finally, the third inequality being an equality implies that $\hat{\mu}$-a.s., $S(\hat{x}) = \hind(X)$, and so $\hat{\mu}$ is supported on $\Xmax$.




\end{proof}


\begin{theorem}\label{mmecreation}
$\Phi$ is an injective map which sends shift-invariant measures supported on $\widehat{\Xinf}_{\max}$ to measures of maximal entropy on $X^{\otimes \infty}$.
Also, $\Phi$ preserves $\mathbb{P}$-invariance of measures.
\end{theorem}

\begin{proof}
Injectivity of $\Phi$ follows from the easily checked fact that $\phi_{\Phi(\hat{\mu})} \circ \Phi(\hat{\mu}) = \hat{\mu}$ for any measure $\hat{\mu}$ supported on $\widehat{\Xinf}_{\max}$.

For any such $\hat{\mu}$, define $\mu := \Phi(\hat{\mu})$. Then $\displaystyle h(\mu) = h(\mu \mid \phi_{\mu}) = \int_{\widehat{X^{\otimes \infty}}} \log |\hat{x}_0|\ d\hat{\mu}(\hat{x})$, but since $\hat{\mu}$ is supported on $\widehat{\Xinf}_{\max}$,
$\displaystyle \int_{\widehat{X^{\otimes \infty}}} \log |\hat{x}_0| \ d\hat{\mu}(\hat{x}) = \hind(\Xinf)$.
Thus, $h(\mu)=\hind(\Xinf)$,which is equal to $h(\Xinf)$ by Theorem~\ref{hind=hinf}, so $\mu$ is a measure of maximal entropy.

If $\hat{\mu}$ is $\mathbb{P}$-invariant, and $\rho \in \mathbb{P}$, then
$$\mu\circ \rho = \int_{\widehat{X^{\otimes \infty}}} \mu_{\hat{x}} \ d\hat{\mu}(\rho \hat{x}) =
\int_{\widehat{X^{\otimes \infty}}} \mu_{\hat{x}} \ d\hat{\mu}(\hat{x})
=\mu,$$ so $\mu$ is $\mathbb{P}$-invariant.

\end{proof}

Theorems~\ref{mmestructure} and \ref{mmecreation} show that there is a bijective correspondence
between the $\mathbb{P}$-invariant measures of maximal entropy on
$X^{\otimes \infty}$ and the $\mathbb{P}$-invariant measures supported on $\Xmax$.

We now recall some technical facts, shown in \cite{LMP}, about independence entropy for
$\mathbb{Z}$ SFTs. It was shown in
Theorem 2 from \cite{LMP} that for any $\mathbb{Z}$ SFT $X$, there
exists a word $\hat{w} \in \mathcal{L}(\hat{X},[1,m])$ such that $\hat{w}^{\infty}
\in \hat{X}$ and $\frac{S(\hat{w})}{m} = \hind(X)$. (Here and in the sequel, for any word $\hat{w}$, $\hat{w}^n$ represents the $n$-fold concatenation $\hat{w} \ldots \hat{w}$, and $\hat{w}^{\infty}$ represents the biinfinite concatenation of infinitely many copies of $\hat{w}$.)
Any such $\hat{w}^{\infty}$ is called a \textit{maximizing cycle} for $\hat{X}$. A word $\hat{w} \in \LL(\hat{X})$ is called a \emph{maximizing word}
for $\hat{X}$ if $\hat{w}^{\infty}$ is a maximizing cycle for $\hat{X}$.

For completeness, we give a self-contained proof of the following refinement of the above statement, which is essentially in \cite{LMP}.

\begin{lemma}\label{simplecycle}
For any $k$-step $\mathbb{Z}$ SFT $X$, $\hat{X}$ has a maximizing word $\hat{w}$ with no repeated $k$-letter subword.
\end{lemma}

\begin{proof} Consider a maximizing word $\hat{w}$ for $\hat{X}$, and denote its length by $m$. If $\hat{w}$ has no repeated $k$-letter subword, we are done. Suppose then that $\hat{w}$ does have a repeated $k$-letter subword, call it $u$. Let's say that $\hat{w}_{[i,i+k-1]} = \hat{w}_{[j,j+k-1]} = u$, $i < j$. Define $\hat{a} = \hat{w}_{[i,j-1]}$ and $\hat{b} = \hat{w}_{[j,m]} \hat{w}_{[1,i-1]}$. We now claim that both $\hat{a}$ and $\hat{b}$ are also maximizing words for $\hat{X}$. Note that every $k$-letter subword of $\hat{a}^{\infty}$ was already a subword of $\hat{w}_{[i,j+k-1]}$, which is clearly in $\LL(\hat{X})$ since $\hat{w}$ is. Similarly, every $k$-letter subword of $\hat{b}^{\infty}$ was already a subword of $\hat{w}_{[k,m]} \hat{w}_{[1,i+k-1]}$, which is in $\LL(\hat{X})$ since $\hat{w}^2$ is. Therefore, both $\hat{a}^{\infty}$ and $\hat{b}^{\infty}$ are in $\hat{X}$.

Finally, we note that since each letter of $\hat{w}$ is contained in exactly one of $\hat{a}$ or $\hat{b}$, $S(\hat{w}) = \hind(X)$ is a weighted average of $S(\hat{a})$ and $S(\hat{b})$. Both $S(\hat{a}) = S(\hat{a}^{\infty})$ and $S(\hat{b}) = S(\hat{b}^{\infty})$ are less than or equal to $\hind(X)$ by definition, so both are equal to $\hind(X)$, and therefore both $\hat{a}^{\infty}$ and $\hat{b}^{\infty}$ are maximizing cycles for $\hat{X}$. Then $\hat{a}$ and $\hat{b}$ are maximizing words for $\hat{X}$, and since each has length less than $\hat{w}$, we can continue this procedure until we arrive at a maximizing word for $\hat{X}$ with no repeated $k$-letter subwords.
\end{proof}

We say that a maximizing word for a $k$-step $\mathbb{Z}$ SFT is \textit{simple} if it has no repeated $k$-letter subwords,
and that a maximizing cycle is simple if it can be written as $\hat{w}^{\infty}$ for a simple maximizing word $\hat{w}$.
Now, for any simple maximizing word $\hat{w}$, we will construct a
specific finite orbit contained in $\Xmax$. The method
is simple: define $x(\hat{w}) \in
\hat{\Sigma}^{\mathbb{Z}^{\infty}}$ by $\displaystyle x(\hat{w})_g = \hat{w}_{\sum
g_i \pmod {|\hat{w}|}}$. Then for any $d,m$ and $g \in \mathbb{Z}^{\infty}$,
$x(\hat{w})_{g + m e_d} = \hat{w}_{\sum g_i + m \pmod {|\hat{w}|}}$, and so
$x(\hat{w})_{g + \zz e_d}$ is just a shift of the sequence
$\hat{w}^{\infty}$. Clearly, this implies that $x(\hat{w}) \in
\Xmax$. Denote by $\mathcal{O}(\hat{w})$ the finite ($\mathbb{P}$-invariant)
orbit of $x(\hat{w})$ under $\mathbb{Z}^{\infty}$, and by $\hat{\mu}_{\hat{w}}$ the
uniform measure on $\mathcal{O}(\hat{w})$. Then by
Theorem~\ref{mmecreation}, $\Phi(\hat{\mu}_{\hat{w}})$ is a $\mathbb{P}$-invariant measure of maximal
entropy on $X^{\otimes \infty}$. 

\begin{theorem}\label{uniqueness}
For any $\mathbb{Z}$ SFT $X$, there is a unique $\mathbb{P}$-invariant measure of maximal entropy on $X^{\otimes \infty}$ if and only if the following two conditions are satisfied:\\

\noindent
{\rm (1)} $\hat{X}$ has a unique (up to shifts) simple maximizing cycle $\hat{w}^{\infty}$

\noindent
{\rm (2)} There is only one finite orbit of points in $\hat{\Sigma}^{\mathbb{Z}^2}$ for which each row and each column is a shift of the sequence $\hat{w}^{\infty}$, namely the orbit of the periodic point $\hat{w}^{(2)}$ defined by $\hat{w}^{(2)}_{(i,j)} = \hat{w}_{i + j \pmod {|\hat{w}|}}$.
\end{theorem}

\begin{proof}
($\Longrightarrow$) If condition (1) is violated, then $\hat{X}$ has two simple
maximizing cycles $\hat{w}^{\infty}$ and $\widehat{w'}^{\infty}$ which are not shifts of each other, which induce
points $x(\hat{w})$ and $x(\widehat{w'})$ with different
finite orbits $\mathcal{O}(\hat{w})$ and $\mathcal{O}(\widehat{w'})$ contained in $\Xmax$.
Then by injectivity of $\Phi$, $\Phi(\hat{\mu}_{\hat{w}})$ and $\Phi(\hat{\mu}_{\widehat{w'}})$ are distinct $\mathbb{P}$-invariant
measures of maximal entropy on $X^{\otimes \infty}$.\\

If condition (1) is satisfied but condition (2) is violated, then $\hat{X}$ has a unique (up to shifts) simple maximizing cycle $\hat{w}^{\infty}$ and a point $\widehat{w'}^{(2)}$ in $\hat{\Sigma}^{\mathbb{Z}^2}$ whose rows and columns are all shifts of $\hat{w}^{\infty}$, but which is not a shift of $\hat{w}^{(2)}$. We now construct an uncountable family of shift-invariant and $\mathbb{P}$-invariant measures supported on $\widehat{X^{\otimes \infty}}_{\max}$, which by Theorem~\ref{mmecreation} will yield an uncountable family of $\mathbb{P}$-invariant measures of maximal entropy on $X^{\otimes \infty}$. First, for any $\alpha \in (0,0.5)$, define $\eta_{\alpha}$ to be the Bernoulli measure on $\{0,1\}^{\mathbb{N}}$ which gives probability $\alpha$ to $0$ and $1 - \alpha$ to $1$ at each site. Define $\upsilon$ to be the uniform measure on $\{0,1,\ldots,|\hat{w}| - 1\}$. Define a map $\tau$ from $\{0,1,\ldots,|\hat{w}| - 1\}^2 \times \{0,1\}^{\mathbb{N}}$ to $\hat{\Sigma}^{\mathbb{Z}^{\infty}}$ by

\[
\tau(i,j,(u_n))_g := \widehat{w'}^{(2)}_{\left(i + \sum_{\{n \ : \ u_n = 0\}} g_n, \ j + \sum_{\{n \ : \ u_n = 1\}} g_n\right)}.
\]

We first claim that $\tau$ maps to $\widehat{X^{\otimes \infty}}_{\max}$. This is easy to check: by definition, for any $i$, $j$, and $(u_n)$, every row of $\tau(i,j,(u_n))$ is just a row or column of $\widehat{w'}^{(2)}$, which will always be a shift of $\hat{w}^{\infty}$. Now, for any $\alpha \in (0,0.5)$, define $\mu_{\alpha}$ to be the push-forward of $\upsilon \times \upsilon \times \eta_{\alpha}$ under $\tau$. Clearly each $\mu_{\alpha}$ is a measure on $\widehat{X^{\otimes \infty}}_{\max}$. The reader may check that $\mathbb{P}$-invariance of $\mu_{\alpha}$ follows from the fact that $\eta_{\alpha}$ is i.i.d., and that shift-invariance of $\mu_{\alpha}$ follows from the uniformity of $\upsilon$ and the fact that $\widehat{w'}^{(2)}$ is periodic with respect to $(|\hat{w}|,0)$ and $(0,|\hat{w}|)$. All that remains is to show that all $\mu_{\alpha}$ are distinct. For any $\alpha$, define $\nu_{\alpha}$ to be the marginalization of $\mu_{\alpha}$ to $\mathbb{Z}^2 \times \{0\}^{\infty}$. It is clear that $\nu_{\alpha}$ is always a finitely supported measure, which is a shift of $\hat{w}^{(2)}$ with probability $\alpha^2 + (1-\alpha)^2$, a shift of $\widehat{w'}^{(2)}$ with probability $\alpha(1-\alpha)$, and a shift of $\underline{\widehat{w'}^{(2)}}$ with probability $\alpha(1-\alpha)$, where $\underline{\widehat{w'}^{(2)}}$ is obtained from $\widehat{w'}^{(2)}$ by permuting the first and second coordinates. But then since $\hat{w}^{(2)}$ is different from both $\widehat{w'}^{(2)}$ and $\underline{\widehat{w'}^{(2)}}$, and since $\alpha^2 + (1-\alpha)^2$ is injective on $(0,0.5)$, clearly all $\nu_{\alpha}$ are distinct, implying that all $\mu_{\alpha}$ are distinct, and therefore that all $\Phi(\mu_{\alpha})$ are distinct $\mathbb{P}$-invariant measures of maximal entropy on $X^{\otimes \infty}$.\\

($\Longleftarrow$) If conditions (1) and (2) are satisfied, then we will show that any shift-invariant 
measure supported on $\Xmax$ is in fact supported on $\mathcal{O}(\hat{w})$. Clearly the only such measure is $\hat{\mu}_{\hat{w}}$, and Theorem~\ref{mmestructure} then implies that the only $\mathbb{P}$-invariant measure of maximal entropy on $X^{\otimes \infty}$ is $\Phi(\hat{\mu}_{\hat{w}})$.

Suppose that $\hat{\mu}$ is a shift-invariant 
measure on $\Xmax$, and choose any $d \in \mathbb{N}$. Then $S(\hat{\mu}) = \int_{\Xmax} \log |\hat{x}_0| \ d\hat{\mu}(\hat{x}) = \int_{\Xmax} S(\hat{x}) \ d\hat{\mu}(\hat{x})$ (by shift-invariance of $\hat{\mu}$). This integral is equal to $\hind(X)$ by definition of $\Xmax$, and is also clearly equal to $\int_{\Xmax} S(\hat{x}_{\zz e_d}) \ d\hat{\mu}(\hat{x})$. Since $S(\hat{x}_{\zz e_d})$ is bounded from above by $\hind(X)$ for all $\hat{x}$, clearly $S(\hat{x}_{\zz e_d}) = \hind(X)$ $\hat{\mu}$-almost surely. Choose $k$ so that $\hat{X}$ is a $k$-step SFT, and consider any $\hat{u}$ which contains no repeated $k$-letter word and for which $\hat{u}^{\infty} \in \hat{X}$ is not a shift of $\hat{w}^{\infty}$. Since $\hat{w}^{\infty}$ was the unique (up to shifts) simple maximizing cycle for $\hat{X}$, $S(\hat{u}) < \hind(X)$. If $\hat{\mu}([\hat{u}]) > 0$, then for $\hat{\mu}$-a.e. $\hat{x} \in \hat{X}$, $\hat{u}$ occurs within $\hat{x}_{\zz e_d}$ with positive frequency. Denote by $\hat{v}$ the $k$-letter prefix of $\hat{u}$. Then we can decompose any $\hat{x}_{\zz e_d}$ as $\ldots \hat{v} u_{-1} \hat{v} u_0 \hat{v} u_1 \hat{v} \ldots$, where $\hat{v} u_i = \hat{u}$ for a set of integers $i$ of positive density. Then, since $(\hat{v} u_i)^{\infty} \in \hat{X}$ for every $i$, the same argument from Lemma~\ref{simplecycle} shows that $S(\hat{v} u_i) \leq \hind(X)$ for all $i$, and so $S(\hat{x}_{\zz e_d}) < \hind(X)$, a contradiction. So, $\hat{\mu}([\hat{u}]) = 0$. Clearly, this shows that $\hat{\mu}$-a.s., $\hat{x}_{\zz e_d}$ is just a shift of $\hat{w}^{\infty}$. By shift-invariance, for $\hat{\mu}$-a.e. $\hat{x} \in \widehat{X^{\otimes \infty}}$ it is also the case that for any $g \in \zz^{\infty}$ and $d \in \mathbb{N}$, $\hat{x}_{g + \zz e_d}$ is also a shift of $\hat{w}^{\infty}$.

Consider any such $\hat{x}$, where every row in every direction is a shift of $\hat{w}^{\infty}$. Then for any dimensions $d_1 < d_2$ and any $g \in \mathbb{Z}^{\infty}$, consider the infinite two-dimensional point $\hat{x}_{g + \mathbb{Z} e_{d_1} + \mathbb{Z} e_{d_2}}$. Each row and column of $\hat{x}_{g + \mathbb{Z} e_{d_1} + \mathbb{Z} e_{d_2}}$ is a shift of $\hat{w}^{\infty}$. But then by condition (2), $\hat{x}_{g + \mathbb{Z} e_{d_1} + \mathbb{Z} e_{d_2}}$ is a shift of $\hat{w}^{(2)}$ and so is periodic with respect to $e_{d_1} - e_{d_2}$. Since this is true for all $d_1$, $d_2$, and $g$, $\hat{x}$ must be periodic with respect to $e_{d_1} - e_{d_2}$ for all $d_1 < d_2$. This implies in turn that $\hat{x}$ is periodic with respect to any $g \in \mathbb{Z}^{\infty}$ with $\sum g_i = 0$. It is simple to check that any point in $\hat{\Sigma}^{\mathbb{Z}^{\infty}}$ which is periodic with respect to all such vectors and whose marginalization to $\mathbb{Z} e_1$ is a shift of $\hat{w}^{\infty}$ must be a shift of $x(\hat{w})$.

We have then shown that any shift-invariant measure $\hat{\mu}$ on $\Xmax$ is supported on $\mathcal{O}(\hat{w})$, which implies that $\Phi(\hat{\mu}_{\hat{w}})$ is the unique $\mathbb{P}$-invariant measure of maximal entropy on $X^{\otimes \infty}$ as explained above.

\end{proof}

The techniques of Theorem~\ref{uniqueness} also allow us to give one more case in which the set of $\mathbb{P}$-invariant measures of maximal entropy can be completely described.

\begin{theorem}\label{disjcycles}
If $X$ is a $\mathbb{Z}$ SFT such that $\hat{X}$ has $k$ different (up to shifts) simple maximizing cycles $(\widehat{w_i})^{\infty}$, $1 \leq i \leq k$, and if no two $\widehat{w_i}$ share a common letter of $\hat{\Sigma}$, and if for each $i \in [1,k]$, there is only one finite orbit of points in $\hat{\Sigma}^{\mathbb{Z}^2}$ for which each row and each column is a shift of the sequence $\widehat{w_i}^{\infty}$, then there are exactly $k$ ergodic $\mathbb{P}$-invariant measures of maximal entropy on $X^{\otimes \infty}$.
\end{theorem}

\begin{proof}

We will only sketch a proof, as the details are almost the same as in the proof of Theorem~\ref{uniqueness}. Firstly, by the same reasoning used there, for any ergodic $\mathbb{P}$-invariant measure of maximal entropy $\mu$ on $\Xinf$ and for any $g \in \zz^{\infty}$ and $d \in \mathbb{N}$, it is $\hat{\mu}$-a.s. the case that $\hat{x}_{g + \zz e_d}$ is a shift of one of the simple maximizing cycles $\widehat{w_i}^{\infty}$. But then since no two $\widehat{w_i}$ share a common letter, this $i$ must be the same for all $g$ and $d$, and by ergodicity, it is $\hat{\mu}$-a.s. constant. So, there exists $i$ for which $\hat{\mu}$-a.s., for any $g \in \zz^{\infty}$ and $d \in \mathbb{N}$, $\hat{x}_{g + \zz e_d}$ is a shift of $\widehat{w_i}^{\infty}$.

Then a similar argument as was used above shows that $\hat{\mu}$ must be supported on $\mathcal{O}(x(\widehat{w_i}))$, and so $\mu = \Phi(\hat{\mu}_{\widehat{w_i}})$. Since there are only $k$ possible choices for $i$, and each clearly gives a different measure, we are done.

\end{proof}

\section{Applications to specific models}\label{sec:application}
The purpose of this section is present some applications of our general results to various specific models which have appeared in the literature.

\subsection{Hard-square model}\label{subsec:hardsquare}
The underlying $\mathbb{Z}$ subshift, also known as the golden mean shift, is
$$\mathcal{H} := \{x \in \{0,1\}^\mathbb{Z}~:~ x_{n}x_{n+1} \ne 11\}.$$
The $d$-dimensional \textit{hard-square model} is then defined as $\mathcal{H}^{\otimes d}$.

It is easily checked that $\hind(\mathcal{H})=\frac{1}{2}\log(2)$ so by our result $h_{\infty}(\mathcal{H})=\frac{1}{2}\log(2)$.
In \cite{LMP}, results of Galvin and Kahn \cite{galvin_kahn04} were used to show directly that $h(\mathcal{H}^{\otimes d}) \to h_{\infty}(\mathcal{H})$ at an
exponential rate, with explicit numerical bounds. It is easily checked that
$(\overline{0\{0,1\}})^{\infty}$ is the unique (up to shifts)
simple maximizing cycle for $\hat{\mathcal{H}}$, and so Theorem~\ref{uniqueness} implies that
there is a unique $\mathbb{P}$-invariant measure of maximal entropy on $\mathcal{H}^{\otimes \infty}$.
In fact by \cite{galvin_kahn04}, uniqueness holds
even without the assumption of $\mathbb{P}$-invariance.
The unique measure of maximal entropy is not weak mixing; since $\mu$-a.s. each point of $\mathcal{H}^{\otimes \infty}$ either
has $0$s on all odd sites or $0$s on all even sites (the parity of $v \in \zz^{\infty}$ is just the parity of the sum of its coordinates), and $\mu$-a.s. only one of these can hold, clearly $\mu$ has an eigenfunction with eigenvalue of $-1$.
The combinatorial methods of Galvin and Kahn
show that this eigenvalue is also present in the (unique) measure of
maximal entropy on $\mathcal{H}^{\otimes d}$ for all sufficiently large $d$.

\subsection{$n$-coloring shifts}

The one-dimensional $n$-coloring shift is
$$\mathcal{C}_n := \{ x \in \{1,\ldots,n\}^\zz~:~ \forall k, \ x_k \ne x_{k+1}\}.$$
The $d$-dimensional \textit{$n$-coloring shift} is defined as $\mathcal{C}_n^{\otimes d}$.

By our results,
 $h_{\infty}(\mathcal{C}_n)=\hind(\mathcal{C}_n)$, which is easily computed to be $\frac{1}{2}\log(\lfloor  n/2 \rfloor)+ \frac{1}{2}\log(\lceil n/2
\rceil)$; there are $n \choose {\lfloor n/2 \rfloor}$ (up to shifts) simple maximizing cycles for $\widehat{\mathcal{C}_n}$, namely
all sequences $(\hat{a} \hat{b})^{\infty}$ for which $\hat{a}$ and $\hat{b}$ form a partition of
$\Sigma$ and $|\hat{a}| = \lfloor \frac{|\Sigma|}{2} \rfloor$. Since no two of these cycles which are not shifts of each other share
a common letter, by Theorem~\ref{disjcycles} there are exactly $n \choose {\lfloor n/2 \rfloor}$
ergodic $\mathbb{P}$-invariant measures of maximal entropy on $h(\mathcal{C}_n^{\otimes \infty})$,
each of which has eigenvalue $-1$ as in the hard-square model.

The case $n=3$ is of particular interest. In \cite{LMP}, it was shown that $h(\mathcal{C}_3^{\otimes d})
\to h_{\infty}(\mathcal{C}_3) = \frac{\log 2}{2}$ exponentially fast. The argument involved creating a correspondence
between configurations in the $d$-dimensional hard-square and $3$-coloring shifts,
and then exploiting the previously mentioned results of Galvin and Kahn.

Also, techniques from recent work of Peled \cite{peled2010} imply that for sufficiently large $d$, $\mathcal{C}_3^{\otimes d}$ has at least $3$ ergodic measures of maximal
entropy. (Compare this with our Theorem~\ref{disjcycles}, which as mentioned above shows that for ``infinite'' $d$, there are exactly $3$ ergodic isotropic measures of
maximal entropy.) We give a rough summary of the (somewhat technical) argument here, and refer the reader to \cite{peled2010} for more details.

Peled's paper is about graph homomorphisms; a graph homomorphism from the graph $G=(V_1,E_1)$ to $H=(V_2,E_2)$ is a function $c \in V_2^{V_1}$, so that $(c_{v},c_w) \in E_2$ whenever $(v,w) \in E_1$. We consider only graph homomorphisms from subgraphs of $\mathbb{Z}^d$ to $\mathbb{Z}$. (We again consider $\mathbb{Z}^d$ as the vertex set of a graph, where $m,n \in \mathbb{Z}^d$ are adjacent iff $\|m-n\|=1$.) Note that each graph homomorphism from a subgraph $G \subseteq \mathbb{Z}^d$ to $\mathbb{Z}$ is a configuration $x \in \mathbb{Z}^{V(G)}$.

Any such graph homomorphism from $G \subseteq \mathbb{Z}^d$ to $\mathbb{Z}$ can be turned into a $3$-coloring on $V(G)$ by the map $\phi$ which coordinatewise sends $n \mapsto n \pmod 3$. If $G$ is the maximal subgraph of $\mathbb{Z}^d$ with vertex set $V(G)$ (in other words, all edges of $\mathbb{Z}^d$ connecting two vertices of $V(G)$ are included), then any $3$-coloring on such $V(G)$ can be ``lifted'' to a graph homomorphism from $G$ to $\mathbb{Z}$ which projects onto the original $3$-coloring under $\phi$. (See Section 4.3 of \cite{schmidt_cohomology_SFT_1995} for details.) In particular, graph homomorphisms from such $G$ to $\mathbb{Z}$ sending any fixed $n\in\mathbb{Z}^d$ to any fixed $t \in \mathbb{Z}$ are in natural bijection to $3$-colorings of $G$ with $n$ colored by $t \pmod 3$.


We need two preliminary results; the first is from \cite{peled2010} and the second was suggested by Peled in a private communication.

\textbf{Fact 1:} There is a sequence $\epsilon_d$ tending to $0$ exponentially fast so that for any $n,d \in \mathbb{N}$, $t \in \mathbb{Z}$, and even site $j \in F_n$ (as in Section~\ref{subsec:hardsquare}, $j$ is even iff $\sum j_i$ is even), if one uniformly chooses a graph homomorphism from $F_{n+1}$ to $\mathbb{Z}$, conditioned on the event that $x_k = t$ for all even boundary sites $k \in F_{n+1} \setminus F_n$, then the probability that $x_j = t$ is greater than $1 - \epsilon_d$.

\textbf{Fact 2:} There is a uniform constant $p > 0$ so that for any $n,d \in \mathbb{N}$, $t \in \mathbb{Z}$, and even site $j \in F_{n+1} \setminus F_n$, if one uniformly chooses a graph homomorphism from $F_{n+1}$ to $\mathbb{Z}$, conditioned on the event that $x_j = t$, then the probability that $x_k = t$ for all even $k \in F_{n+1} \setminus F_n$ is greater than $p^{|F_{n+1} \setminus F_n|}$.

We first turn these into statements about $3$-colorings by using the above described bijections. We note that for any $n$ and $d$ and graph homomorphism $f$ from $F_{n+1}$ to $\mathbb{Z}$, if we fix the value $f_j$ at a single even site $j$ in $F_{n+1} \setminus F_n$ (say $f_j = t$), then the statements ``$f_k = t$ for all even $k \in F_{n+1} \setminus F_n$'' and ``$(\phi f)_k = t \pmod 3$ for all even $k \in F_{n+1} \setminus F_n$'' are equivalent. The forward implication is trivial, and the reverse relies on noticing that if two nearest even vertices (i.e. distance $2$) have the same label in $\phi f$, then their labels in $f$ differ by a multiple of $3$, and the only possible such multiple of $3$ is $0$ since they are labels of vertices at distance $2$ under a graph homomorphism. Therefore, the above statements can be recast as follows:

\textbf{$3$-coloring version of Fact 1:} For any $n,d \in \mathbb{N}$, $t \in \{0,1,2\}$, and even site $j \in F_n$ (as before, $j$ is even iff $\sum j_i$ is even), if one uniformly chooses a $3$-coloring on $F_{n+1}$, conditioned on the event that $x_k = t$ for all even boundary sites $k \in F_{n+1} \setminus F_n$, then the probability that $x_j = t$ is greater than $1 - \epsilon_d$.

\textbf{$3$-coloring version of Fact 2:} There is a uniform constant $p > 0$ so that for any $n,d \in \mathbb{N}$, $t \in \{0,1,2\}$, and even site $j \in F_{n+1} \setminus F_n$, if one uniformly chooses a $3$-coloring on $F_{n+1}$, conditioned on the event that $x_j = t$, then the probability that $x_k = t$ for all even $k \in F_{n+1} \setminus F_n$ is greater than $p^{|F_{n+1} \setminus F_n|}$.

Now, we define some measures on $\mathcal{C}_3^{\otimes d}$. For any fixed $n, d \in \mathbb{N}$ and $t \in \{0,1,2\}$, define $\nu_{n,d,t}$ to be the uniform measure over all $3$-colorings of $F_{n+1}$ which have label $t$ at all even vertices in $F_{n+1} \setminus F_n$; by the $3$-coloring version of Fact 1, $\nu_{n,d,t}(f_j = t) > 1 - \epsilon_d$ for any even $j \in F_n$. Then define $\eta_{n,d,t}$ to be the measure on $\{0,1,2\}^{\mathbb{Z}^d}$ in which each shifted cube $(F_{n+1}) + (2n+3) v$ ($v \in \mathbb{Z}^d$) is independently assigned according to $\nu_{n,d,t}$; then $\eta_{n,d,t}(f_{j + (2n + 3)v} = t) > 1 - \epsilon_d$ for any even $j \in F_n$ and $v \in \mathbb{Z}^d$. Also, note that the support of $\eta_{n,d,t}$ is not contained in $\mathcal{C}_3^{\otimes d}$. Finally, define $\mu_{n,d,t}$ to be $\frac{1}{|F_{n+1}|} \sum_{v \in F_{n+1}} \sigma_v \eta_{n,d,t}$; clearly $\mu_{n,d,t}$ is an invariant measure, and $\mu_{n,d,t}(f_{j + (2n + 3) v} = t) > \frac{0.5 |F_n| (1 - \epsilon_d) + 0.5 |F_{n+1} \setminus F_n|}{|F_{n+1}|} > 0.5 - 0.5 \epsilon_d$ for any $j \in \mathbb{Z}^d$. Finally, we define $\mu_{d,t}$ to be any weak limit of a subsequence of $\mu_{n,d,t}$ as $n \rightarrow \infty$. It should be clear that $\mu_{d,t}$ is a measure on $\mathcal{C}_3^{\otimes d}$, and that $\mu_{d,t}(f_j = t) \geq 0.5 - 0.5 \epsilon_d$ for any $j \in \mathbb{Z}^d$. If we write $t'$ and $t''$ for the elements of $\{0,1,2\}$ which are not $t$, then it should be clear that $\mu_{d,t}$ is invariant under coordinatewise switching $t'$ and $t''$, and so $\mu_{d,t}(f_j = t') = \mu_{d,t}(f_j = t'') = 0.5(1 - \mu_{d,t}(f_j = t)) \leq 0.25 + 0.25 \epsilon_d$ for every $j \in \mathbb{Z}^d$. This implies that $\mu_{t,d}$, $t = 0,1,2$, are distinct measures as long as $d$ is chosen so large that $\epsilon_d < \frac{1}{3}$, since then $0.5 - 0.5 \epsilon_d > 0.25 + 0.25 \epsilon_d$.

Finally, we claim that each of the three $\mu_{d,t}$ is a measure of maximal entropy. To see this, note that by the $3$-coloring version of Fact 2, the number $N_{n,t,d}$ of $3$-colorings of $F_{n+1}$ with $x_k = t$ for all even $k \in F_{n+1} \setminus F_n$ is greater than $p^{|F_{n+1} \setminus F_n|}$ times the number of $3$-colorings of $F_{n+1}$ with $x_j = t$ for a fixed even site $j \in F_{n+1} \setminus F_n$, and that this is clearly exactly $\frac{1}{3}$ of the total number $\mathcal{L}(\mathcal{C}_3^{\otimes d}, F_{n+1})$ of $3$-colorings of $F_{n+1}$. Therefore,
$\lim_{n \rightarrow \infty} h(\mu_{n,t,d})$ is equal to
\[
\lim_{n \rightarrow \infty} \frac{1}{|F_{n+1}|} \log N_{n,t,d} \geq \limsup_{n \rightarrow \infty} \frac{1}{|F_{n+1}|} \log \left(\frac{1}{3} p^{|F_{n+1} \setminus F_n|} |\mathcal{L}(\mathcal{C}_3^{\otimes d}, F_{n+1})|\right)
\]
\[
= \limsup_{n \rightarrow \infty} \frac{|F_{n+1} \setminus F_n| \log p}{|F_{n+1}|} - \frac{\log 3}{|F_{n+1}|} +
\frac{\log |\mathcal{L}(\mathcal{C}_3^{\otimes d}, F_{n+1})|}{|F_{n+1}|} =
\lim_{n \rightarrow \infty} \frac{\log |\mathcal{L}(\mathcal{C}_3^{\otimes d}, F_{n+1})|}{|F_{n+1}|},
\]
which is equal to $h(\mathcal{C}_3^{\otimes d})$, and since $\mu_{t,d}$ is a weak limit of a subsequence of $\mu_{n,t,d}$, upper semi-continuity of entropy implies that $h(\mu_{t,d}) = h(\mathcal{C}_3^{\otimes d})$, as claimed. Note that the same argument which proved that $\mu_{0,d}$, $\mu_{1,d}$, and $\mu_{2,d}$ are distinct for large $d$ also shows that none is a linear combination of the other two, and so there must be at least three ergodic measures of maximal entropy on $\mathcal{C}^{\otimes d}$ for large enough $d$.

\subsection{Beach model}

In \cite{burton_steif94},
Burton and Steif defined the \textit{$d$-dimensional beach model},
for any $M > 0$, to be the nearest-neighbor $\mathbb{Z}^d$ SFT on the alphabet $\{-M,\ldots,-1,\newline1,\ldots,M\}$
defined by the restriction that adjacent letters must have product greater than or equal to $-1$.
In other words, a negative and positive cannot be adjacent in any cardinal direction unless they are $1$ and $-1$.
These are clearly all axial powers of the one-dimensional beach model, which we denote by $B_M$.

It is easy to show that for $M>2$, $\hind(B_M) = \log M$ and there are exactly two simple maximizing cycles for $B_M$, namely $\{-M,\ldots,-1\}^{\infty}$ and $\{1,\ldots,M\}^{\infty}$. (When $M=1$, $B_M$ is just the full shift on two symbols, and when $M=2$, there is an additional maximizing cycle $\{-1,1\}^{\infty}$. We will not address these special cases further here.) Therefore, our results show that $h(B_M^{\otimes d}) \rightarrow \log M$, and that $B_M^{\otimes \infty}$ has exactly two ergodic $\mathbb{P}$-invariant measures of maximal entropy.

In fact, it was also shown in \cite{burton_steif94}
that for any fixed $d$ and for $M > 4e 28^d$, $B_M^{\otimes d}$
has exactly two ergodic measures of maximal entropy.
Our result implies the same fact for any fixed $M > 2$ and infinite $d$, which suggests that perhaps this is true for any fixed $M$ and large enough $d$, and in fact this was stated as a conjecture in \cite{burton_steif_survey}.




\subsection{Run-length limited shifts}

For any $0 \leq d < k \leq \infty$, the \textit{$(d,k)$ run-length limited shift}, also denoted by $RLL(d,k)$, is the SFT on the alphabet $\{0,1\}$ consisting of all sequences in which all maximal ``runs'' of $0$s have length inside the interval $[d,k]$. For instance, $RLL(0,\infty)$ is the full shift on two symbols, and $RLL(1,\infty)$ is the usual golden mean shift. For any $0 \leq d < k < \infty$, it was shown in \cite{CMP} that
\[
\hind(RLL(d,k)) = \frac{\lfloor (k-d)/(d+1) \rfloor \log 2}{\lfloor (k+1)/(d+1) \rfloor (d+1)} \textrm{ and}
\]
\[
\hind(RLL(d,\infty)) = \frac{\log 2}{d+1}.
\]

It was shown in \cite{OR}, using combinatorial methods, that for any $d$, $h(RLL(d,\infty)^{\otimes d}) \rightarrow \hind(RLL(d,\infty))$, and that the rate is exponential. Our results show that this same convergence is true for any $d$ and $k$ (but say nothing about the rate.)

It is relatively simple to check that there is a unique (up to shifts) simple maximizing cycle for any $d$ and $k$. (This was essentially done, without using our terminology, in \cite{CMP}.) These cycles are given by the maximizing words
\[
\hat{w} = 0^d \{0,1\} \textrm{ for } RLL(d,\infty) \textrm{ and}
\]
\[
\hat{w} = 0^d 1 (0^d \{0,1\})^{\lfloor (k-d)/(d+1) \rfloor} \textrm{ for } RLL(d,k).
\]

For each of these maximizing cycles $\hat{w}^{\infty}$, the reversed version $(\hat{w}_{|\hat{w}|} \hat{w}_{|\hat{w}|-1} \ldots \hat{w}_1)^{\infty}$ is just a shift of $\hat{w}^{\infty}$. 
Therefore, for any $RLL(d,k)$, we can construct a point $\widehat{w'}^{(2)}$ defined by $\widehat{w'}^{(2)}_{(i,j)} = \hat{w}_{i - j \pmod {|\hat{w}|}}$, in which all rows and columns are shifts of $\hat{w}^{\infty}$. In all cases except $RLL(1,\infty)$ and $RLL(0,1)$ (for which the associated simple maximizing words have length $2$) and $RLL(0,\infty)$ (for which the associated simple maximizing word has length $1$), $\widehat{w'}^{(2)}$ is not equal to the point $\hat{w}^{(2)}$ from Lemma~\ref{uniqueness}, and so if $X$ is any run-length limited shift except $RLL(1,\infty)$, $RLL(0,1)$, or $RLL(0,\infty)$, then $\Xinf$ has multiple $\mathbb{P}$-invariant measures of maximal entropy. For each of these three special cases, $\Xinf$ has a unique $\mathbb{P}$-invariant measure of maximal entropy, which we already knew, as these are simply the golden mean shift, the golden mean shift with digits $0$ and $1$ switched, and the full shift on two symbols, respectively.

\subsection{Even shift}

The \textit{even shift} is the $\zz$ subshift $X_{\mathcal{E}}$ on the alphabet $\{0,1\}$ consisting of all sequences
in which all maximal ``runs'' of $0$s have even length. It is easy to show that
$\hind(X_{\mathcal{E}}) = 0$; in fact, $\hat{\Sigma} = \{\{0\},\{1\},\{0,1\}\}$, and it is not hard to check that
the maximum number of times the symbol $\{0,1\}$ can appear in a point of $\hat{X_{\mathcal{E}}}$ is two.
So, our results imply that $h_{\infty}(X_{\mathcal{E}}) = 0$.

Finding $h_{\infty}(X_{\mathcal{E}})$ was of particular interest to us for two reasons.
Firstly, in \cite{LM}, a combinatorial argument was used to show that for the similarly defined
odd shift $\mathcal{O}$, $h(\mathcal{O}^{\otimes d}) = \frac{1}{2}$ for all $d$.
Secondly, it was shown in \cite{LMP} that $h_{\infty}(X) = 0$ for any $\mathbb{Z}$ SFT $X$
with zero independence entropy, and it was naturally wondered if the same was true for \emph{sofic} shifts.
Recall that a shift is \emph{sofic} if it is a factor of an SFT, which is the case for the even shift.

\subsection{Dyck shift}
The \textit{Dyck shift} is the $\zz$ subshift $\mathcal{D} \subset
(\{\alpha_1,\ldots,\alpha_M\}\cup\{\beta_1,\ldots,\beta_M\})^{\zz}$
obtained by considering the $\alpha_i$'s and $\beta_i$'s as $M$ ``types'' of
opening and closing ``brackets'' respectively. The constraints are that matching open and closed brackets must
be of the same ``type,'' which in our terminology means the same subscript.
This interesting non-sofic shift originated
in the study of formal languages. It was introduced into symbolic dynamics
in \cite{krieger74}, where it was shown that
$h(\mathcal{D})=\log(M+1)$ and that there are exactly $2$ ergodic measures of
maximal entropy.

It is easily shown that for $M > 2$, $\hind(\mathcal{D}) = \log M$, and there are precisely two (up to shifts)
simple maximizing cycles on $\hat{\mathcal{D}}$, namely
$\{\alpha_1,\ldots,\alpha_M\}^{\infty}$ and $\{\beta_1,\ldots,\beta_M\}^{\infty}$.
Thus $h_{\infty}(\mathcal{D}) = \log M$, and since these cycles contain no common letter,
Theorem~\ref{disjcycles} shows that there are precisely two ergodic $\mathbb{P}$-invariant
measures of maximal entropy on $\mathcal{D}^{\otimes \infty}$.

\subsection{Symmetric nearest-neighbor SFTs}
In recent work of Engbers and Galvin (\cite{enbers_galvin2011}), they study, for any
finite undirected graph $\mathbb{H}$, the limiting behavior of the
distribution of uniformly chosen graph homomorphisms from discrete $d$-dimensional $m$-tori
$\mathbb{Z}^d_m = \{1,\ldots,m\}^d$ to $\mathbb{H}$ as $d \to \infty$.
The connection to our work comes from the following fact: for a fixed finite undirected graph $G=(V,E)$, the collection of graph homomorphisms from $\mathbb{Z}^d$ to $G$ is precisely the nearest-neighbor SFT $X_G \subset V^{\mathbb{Z}^d}$ defined by enforcing $(x_{n},x_{n+e_i}) \in E$ for all $n \in \mathbb{Z}^d$ and $i=1 \ldots d$, which is a $d$-dimensional axial power of a symmetric nearest-neighbor SFT.

The authors prove that for any fixed $m$ and undirected finite graph $\mathbb{H} = (\mathbb{V}, \mathbb{E})$,
with probability tending to $1$ as $d \to \infty$, a uniformly chosen
random graph homomorphism $x$ from $\mathbb{Z}^d_m$ to $\mathbb{H}$ has corresponding disjoint $A,B \subset \mathbb{V}$
which induce a complete bipartite graph, with $|A||B|$ maximal, such that $x_n \in A$ for
a large proportion of even vertices $n \in \mathbb{Z}^d_m$ and $x_m \in B$ for a large proportion of odd vertices $m \in \mathbb{Z}^d_m$.

These results are related to ours in that they provide an alternative
proof for some of our results for the particular case of symmetric
nearest-neighbor SFTs. The methods of \cite{LMP} show that for any symmetric nearest-neighbor
SFT $X \subset \Sigma^\mathbb{Z}$, there is always a simple maximizing word in $\widehat{X}$ of length $2$, and that
any such word corresponds to a maximal induced complete bipartite graph of $\mathbb{H}$.
In contrast to this paper, the methods of Engbers and Galvin are finitistic, and do not directly invoke exchangeability. In comparison to our results, they provide further quantitative information about the rate of convergence. On the other hand, their work does not seem to directly give any results in the case of non-nearest-neighbor and non-symmetric SFTs.






\section{Further Problems and research directions}\label{future}

Here we summarize a few possible directions for extensions or generalizations of our results.

\subsection{Pressure and equilibrium measures}
In statistical mechanics, it is common to introduce a ``potential'' or ``activity function,'' in which case the role of topological entropy is replaced by topological pressure, and measures of maximal entropy are generalized to equilibrium measures. From the ergodic-theoretic point of view, many results generalize without difficulty (see, for example, \cite{walters_var} for a statement and proof of the variational principle for pressure).

In \cite{enbers_galvin2011} and \cite{galvin_kahn04}, for some specific $\mathbb{Z}$ subshifts, an analysis of equilibrium measures was carried out with respect to a single-site potential on $X^{\otimes d}$ as $d \to \infty$.

It is rather easy to generalize the statements and proofs of our results to the setting where entropy is replaced by pressure with respect to a so-called ``single-site potential'' $f:X \to \mathbb{R}$, given by $f(x)=g(x_0)$ for some function $g:\Sigma \to \mathbb{R}$. From the statistical mechanics viewpoint, $f$ involves no ``interactions'' between sites.

Here is a brief formulation of the analogous result:
For $\hat{w} \in \mathcal{L}(\hat{X},F)$, define $\displaystyle S_f(\hat{w})=\frac{1}{F}\sum_{n \in F} \log \sum_{a \in \hat{w}_n} g(a)$.
Define $P_{ind}(X,f)$ analogously to $\hind(X)$, with $S_f$ replacing $S$,
and
$$P_{\infty}(X,f) = \lim_{d \to \infty} P(X^{\otimes d},f^{(d)}),$$
where $f^{(d)}:X^{\otimes d} \to \mathbb{R}$ is again given by $f(x)=g(x_0)$.

By following our proof of Theorem \ref{hind=hinf}, one can deduce that $P_{ind}(X,f)=P_{\infty}(X,f)$.





\subsection{Finitistic results}

Our techniques involve studying the system $\Xinf$, which is a sort of ``infinite-dimensional'' axial power of $X$. It is natural to wonder what information we can glean about the finite-dimensional axial powers $X^{\otimes d}$. For instance, we have shown that $h(X^{\otimes d}) \rightarrow \hind(X)$, but with no information about the rate. This is a question of particular interest, since it was noted in \cite{LMP} that for all examples where the rate of convergence is known, this rate is exponential.

Another example of useful finite-dimensional information regards measures of maximal entropy. As described in Section~\ref{sec:application}, our results allow us to count the $\mathbb{P}$-invariant ergodic measures of maximal entropy on $\Xinf$ for many models, such as the hard-square model and $n$-coloring shift. It is natural to assume that such results should allow us to draw conclusions about the number of $\mathbb{P}$-invariant ergodic measures of maximal entropy on $X^{\otimes d}$ for large enough $d$, but we have not yet been able to do so.

One reason why we believe such finitistic results should be provable is that one of the keys to our proofs, de Finetti's Theorem, has versions which apply to finite sets of exchangeable random variables \cite{MR577313}. We hope to use these finite versions to answer some finitistic questions in future work.


One specific case in which exponential convergence of $h(X^{\otimes d})$ to $\hind(X)$ has been proven is in the case where $X$ is an SFT and $\hind(X) = 0$. (\cite{LMP}) Interestingly, in the case where $X$ is nearest-neighbor, this convergence is trivial.

\begin{lemma}
For any nearest-neighbor $\mathbb{Z}$ SFT $X$ with zero independence entropy,
$h(X^{\otimes 2}) = 0$.
\end{lemma}
\begin{proof}
Let $X \subset \Sigma^\mathbb{Z}$ be a n.n. SFT with zero independence entropy.
We can restrict to the non-wandering part of $X$, which is a disjoint union of irreducible
subshifts of finite type, thus we may assume $X$ is irreducible without loss of generality.

First we show that for any $a,b \in \Sigma$,
there is at most one $c \in \Sigma$ such that $acb \in \LL(X)$. If
this were not the case, there would exist distinct $c_1,c_2 \in \Sigma$ with $ac_ib \in \LL(X)$ for $i=1,2$,
and then for some $d_1,\ldots,d_k \in \Sigma$, $\{a\}\{c_1, c_2\} \{b\} \{d_1\} \cdots \{d_k\} \{a\}$ would be
in $\LL(\hat{X})$, which would yield a sequence $(\{a\}\{c_1,c_2\}\{b\}\{d_1\}\cdots\{d_k\})^{\infty}$ with positive independence score, contradicting $h_\mathit{ind}(X)=0$.

Next we claim that
for any finite $F \subset \mathbb{Z}^2$ and any $\delta \in \Sigma^{\partial F}$
there is at most one $X^{\otimes 2}$-admissible configuration
$u \in A^{F \cup \partial F}$ with $u_{\partial F}=\delta$; this is done by induction on $|F|$.

The base case of the induction is $|F|=1$; say $F=\{(n,m)\}$. Then $(n-1,m),(n,m+1) \in \partial F$.
Let $a=\delta_{(n-1,m)}$ and $b=\delta_{(n,m+1)}$: it follows that there is at most one $c$ for which $acb \in \LL(X)$.
Since $X$ is nearest-neighbor, this means that there is at most one $c$ for which $ac$
and $cb$ are both $X$-admissible. Thus, for every letter $d \neq c$, $\delta d$ is not $X^{\otimes 2}$-admissible,
implying that there is at most one $X^{\otimes 2}$-admissible filling of $F \cup \partial F$ given $\delta$.

For the inductive step, choose $(n,m) \in F$ so that $(n-1,m),(n,m+1) \in \partial F$. For instance,
take $(n,m)$ to be the leftmost site in the topmost row of $F$.
By the same argument, there is at most one letter $c_{(n,m)}$ which can fill $(n,m)$ in a $X^{\otimes 2}$-admissible
way given $\delta_{(n-1,m)}$ and $\delta_{(n,m+1)}$.
Now the induction hypothesis on $F \setminus \{(n,m)\}$ implies that
there is at most one $X^{\otimes 2}$-admissible filling of $F \cup \partial F$ given $\delta$ and $c_{(n,m)}$ and,
by definition of $c_{(n,m)}$, at most one $X^{\otimes 2}$-admissible filling of $F \cup \partial F$ given $\delta$.

This implies that $|\mathcal{L}(X^{\otimes 2}, [1,n]^2)| \leq |\Sigma|^{|\partial [1,n]^2|}$ for any $n$, and so $h(X^{\otimes 2}) = 0$.

\end{proof}

For subshifts of finite type which are not nearest-neighbor, there is no finite $d$ for which
zero independence entropy implies $h(X^{\otimes d})=0$; it was demonstrated in \cite{IKNZ99} that for any $d > 1$,
$h(RLL(n,k)^{\otimes d})=0$ iff $k=n+1$, yet $\hind(RLL(n,k))=0$ if $k\le 2n$.

\bibliographystyle{abbrv}
\bibliography{limiting_entropy}
\end{document}